\documentclass[12pt]{amsart}

\usepackage{amsmath,amsthm,amssymb}
\usepackage{tikz}
\usetikzlibrary{arrows.meta}

\usepackage{a4wide}
\usepackage{enumerate}
\usepackage{multicol, multirow}
\newtheorem{theorem}{Theorem}[section]
\newtheorem{lemma}[theorem]{Lemma}
\newtheorem{proposition}[theorem]{Proposition}
\newtheorem{corollary}[theorem]{Corollary}
\newtheorem*{MainTheorem}{Main Theorem}
\usepackage{hyperref}

\theoremstyle{remark}
\newtheorem{remark}[theorem]{Remark}

\newtheorem*{claim*}{Claim}

\newcommand{\C}{\ensuremath{\mathbb{C}}}
\newcommand{\R}{\ensuremath{\mathbb{R}}}

\newcommand{\g}[1]{\ensuremath{\mathfrak{#1}}}

\newcommand{\II}{\ensuremath{I\!I}}
\DeclareMathOperator{\tr}{tr}

\DeclareMathOperator{\id}{Id}
\DeclareMathOperator{\Ad}{Ad}

\DeclareMathOperator{\ad}{ad}
\DeclareMathOperator{\Exp}{Exp}

\DeclareMathOperator{\spann}{span}

\DeclareMathOperator{\Ric}{Ric}

\newcommand{\tabitem}{~~\llap{\textbullet}~~}

\newcommand{\Ss}{\ensuremath{\mathcal{S}}}

\begin{document}
\title[Ricci solitons as submanifolds of $\C H^n$]{Ricci solitons as submanifolds of \\ complex hyperbolic~spaces}
%    author one information
\author[A.~Cidre-D\'iaz]{\'Angel~Cidre-D\'iaz}
\address{CITMAga, 15782 Santiago de Compostela, Spain.\newline\indent Department of Mathematics, Universidade de Santiago de Compostela, Spain}
\email{angel.cidre.diaz@usc.es}
%    author two information
\author[V.~Sanmart\'in-L\'opez]{V\'ictor Sanmart\'in-L\'opez}
\address{CITMAga, 15782 Santiago de Compostela, Spain.\newline\indent Department of Mathematics, Universidade de Santiago de Compostela, Spain}
\email{victor.sanmartin@usc.es}

\begin{abstract}
Any homogeneous expanding Ricci soliton is known to be isometric to a Lie subgroup of the solvable part of the Iwasawa decomposition associated with a symmetric space of non-compact type, with the metric induced as a submanifold. In this paper, we classify and analyze the geometry of such Lie subgroups with Ricci soliton induced metric when the symmetric spaces are complex hyperbolic~spaces.
\end{abstract}

\thanks{The authors have been supported by Grant PID2022-138988NB-I00 funded by MICIU/AEI/10.13039/501100011033 and by ERDF, EU, by ED431F 2020/04 and ED431C 2023/31 (Xunta de Galicia, Spain). The first author acknowledges support of a FPU fellowship. 
} 
\maketitle
%%%%%%%%%%%%%%%%%%%%%%%%%%%%%%%%%%%%%%%%%%%%%%%%%%%%%%%%%%%%%%%%%%%%%%%%%%

\section{Introduction}
In this paper, we focus on the intersection between submanifold geometry in complex hyperbolic spaces with homogeneous Ricci soliton and Einstein metrics.

On the one hand, a Riemannian manifold $(M, g)$ is said to be a \emph{Ricci soliton} if its Ricci tensor, $\Ric$, reads as $\Ric = c g + \mathcal{L}_X g$, where $c$ is a real number and $\mathcal{L}_X$ denotes the Lie derivative with respect to a smooth vector field $X$ of $M$. Ricci solitons generalize \emph{Einstein metrics}, i.e.,  those whose Ricci tensor is a real multiple of the metric. The main source of interest on Ricci solitons stems from the fact that they are generalized fixed points of the Ricci-flow. This tool was introduced by Hamilton in~\cite{Ha82}, in order to evolve and improve metrics in his program to address the Thurston geometrization conjecture. Indeed, Ricci soliton metrics only dilate under the Ricci-flow, up to diffeomorphism, and such dilation is encoded in the cosmological constant $c$, so they are called \emph{steady} ($c = 0$), \emph{shrinking} ($c > 0$) or \emph{expanding} ($c < 0$). Later, the Ricci-flow was successfully utilized by Perelman to solve the Poincaré conjecture. Since then, Ricci solitons were deeply investigated in the literature, assuming many times homogeneity (see the survey~\cite{La09}), which converts the non-linear second order~PDE defining a Ricci soliton into an algebraic equation.

On the other hand, many investigations have been developed in the context of submanifold geometry of complex hyperbolic spaces over the last few years. We emphasize the following outstanding classifications: Hopf hypersurfaces with constant principal curvatures~\cite{Be89}, cohomogeneity one actions~\cite{BB01, BeTa07}, and isoparametric hypersurfaces~\cite{DDS:adv}. Some other relevant classifications addressed in these spaces are those of polar actions~\cite{DDK17}, homogeneous Lagrangian foliations~\cite{DDH} or hypersurfaces with constant principal curvatures under certain extra hypotheses~\cite{DD11}. See \cite{Jurgen, CeRy} for more information on submanifolds in complex space forms, or~\cite{Gold} for more information in complex hyperbolic geometry. 

Despite the wide range of submanifolds that have been investigated in complex hyperbolic spaces, these spaces do not admit Einstein hypersurfaces~\cite[Theorem~8.69]{CeRy}. Indeed, Einstein hypersurfaces are rare in symmetric spaces. Fialkow classified them in space forms~\cite{Fi38}, and they are totally umbilical hypersurfaces, hypersurfaces with index of relative nullity one (see~\cite[Chapter~7]{DaTo} for more details), or the product of spheres of the same Ricci curvature for the case of spheres. This result was extended later to the rest of the rank one symmetric spaces in several works by different authors (see~\cite[Theorem~1]{NiPa21} for a summary), and the only examples are geodesic spheres of certain radii, one in quaternionic projective spaces and one in the Cayley projective plane. In 2023, Nikolayevsky and Park completed the classification of Einstein hypersurfaces in irreducible symmetric spaces by analyzing those of rank~$\geq 2$~\cite{NiPa21}. There, the picture is the opposite to the rank one case, as there is a family of minimal Einstein hypersurfaces in symmetric spaces of non-compact type, and just one more example (non-complete and hence non-homogeneous) in $SL(3, \R)/SO(3)$ and in $SU(3)/SO(3)$, respectively.

Homogeneous Ricci solitons have been also investigated from this submanifold perspective. As we illustrate right after the statement of the Main Theorem, connected expanding homogeneous Ricci solitons lie, up to isometry, into a particular class of homogeneous submanifolds of symmetric spaces of non-compact type. Therefore, within this class of homogeneous submanifolds, it is natural to look for those with a Ricci soliton induced metric. With this approach, Tamaru~\cite{Ta11} constructed new examples of Einstein solvmanifolds, which are minimal from the submanifold viewpoint, and of Ricci solitons. Moreover, classifications by imposing conditions on the codimension~\cite{DST20, S22} were obtained by Domínguez-Vázquez, Tamaru and the second author.

In this paper, we focus on the classification of this particular class of homogeneous submanifolds that are Ricci solitons with the induced metric, when the symmetric spaces are complex hyperbolic spaces. Let us be more precise before stating the main result of this work. Let $G = SU(1,n)$ be the connected component of the identity of the isometry group of the complex hyperbolic space $\C H^n$ up to a finite quotient. Put $\g{g}$ for its Lie algebra, and let $\g{g} = \g{k} \oplus \g{p}$ be the Cartan decomposition of the latter, where $\g{k}$ is the Lie algebra of the isotropy group $K=S(U(1)U(n))$. Let $\g{g} = \g{g}_{-2 \alpha} \oplus \g{g}_{-\alpha} \oplus \g{g}_0 \oplus \g{g}_{\alpha} \oplus \g{g}_{2 \alpha}$ be its root space decomposition with respect to a maximal abelian subspace $\g{a}$ of $\g{p}$. Let $AN$ be the connected (solvable) Lie subgroup of $G$ whose Lie algebra is $\g{a} \oplus  \g{g}_{\alpha} \oplus \g{g}_{2 \alpha}$, which is the semidirect product of an abelian Lie group $A$ and a nilpotent Lie group $N$, with Lie algebras $\g{a}$ and $\g{n} = \g{g}_{\alpha} \oplus \g{g}_{2 \alpha}$, respectively.  Now, we select the metric on $AN$ making it isometric to $\C H^n$, which turns out to be left-invariant. We denote by $J$ the complex structure of $AN$ induced by the one in $\C H^n$. Let $B$ and $Z$ be unit vectors spanning $\g{a}$ and $\g{g}_{2\alpha}$ (which are one-dimensional), respectively, with $J B = Z$. 

The main purpose of this paper is to obtain the classification of the Lie subgroups of $AN \cong \C H^n$ that are Ricci solitons with the induced metric. We state now the main result of the paper. See Remark~\ref{remark:congruence} for more information about the congruence classes of the examples of the classification.

\begin{MainTheorem}\label{MainTheorem}
Let $S$ be a connected Lie subgroup of dimension $\geq 2$ of the solvable part $AN$ of the Iwasawa decomposition of the connected component of the identity of the isometry group of the complex hyperbolic space $\C H^n$. Then, $S$ is a Ricci soliton when considered with the induced metric if and only if one (and only one) of the following conditions holds for its Lie algebra~$\g{s}$ (see Remark~\ref{remark:notation} for the notation):
\begin{table}[h]
\renewcommand*{\arraystretch}{2}
\scriptsize
\scalebox{0.82}{\begin{tabular}{|l|l|l|l|l|}
\hline
\textbf{Item} &\textbf{Subalgebra}  & \textbf{Conditions} & \textbf{Einstein?} & \textbf{$S$ is isometric to} \\
\hline
\multirow{2}{0.01cm}{\begin{minipage}{0.8cm}
\begin{enumerate}[{\rm (i)}]
\item \label{MainTheorem:1}
\end{enumerate}
\end{minipage}}& \multirow{2}{*}{$\g{s} = \g{m}_{\pi/2} \oplus \R (V + t Z)$}
& \tabitem $\C V \perp \g{m}_{\pi/2}$
&  \multirow{2}{*}{Yes}  & \multirow{2}{*}{Euclidean space}   \\
& & \tabitem   $\dim\g{m}_{\pi/2} \geq 2 - \dim \R (V + t Z)$ &    &    \\
\hline
& &  
\tabitem $\C U$, $\C V \perp \g{m}_{\pi/2}$ &  &  \\
\begin{minipage}{0.8cm}
\begin{enumerate}[{\rm (i)}]
\setcounter{enumi}{1}
\item \label{MainTheorem:2}
\end{enumerate}
\end{minipage}& $\g{s} = \R (B + U + x Z) \oplus \g{m}_{\pi/2} \oplus \R (V + t Z)$ &\quad \tabitem If $V \neq 0$, then $\langle JU,V\rangle=-\frac{t}{2}$& Yes & Real hyperbolic space \\
& &\qquad \tabitem If $V = 0$, either $t = 0$ or $\g{m}_{\pi/2} = 0$& &	\\	
\hline
\begin{minipage}{0.8cm}
\begin{enumerate}[{\rm (i)}]
\setcounter{enumi}{2}
\item \label{MainTheorem:3}
\end{enumerate}
\end{minipage}&$\g{s} = \R(B + U) \oplus \g{m}_{\varphi} \oplus \g{g}_{2\alpha}$ & \tabitem  $\C U \perp \g{m}_{\varphi}$ \qquad \qquad \tabitem   $|U|=\tan \varphi$ & Yes & Complex hyperbolic space   \\		
\hline
\begin{minipage}{0.8cm}
\begin{enumerate}[{\rm (i)}]
\setcounter{enumi}{3}
\item \label{MainTheorem:4}
\end{enumerate}
\end{minipage} & $\g{s} = \g{m}_{\varphi} \oplus  \g{m}_{\pi/2} \oplus \g{g}_{2\alpha}$ & & No & Heisenberg group $\times \R^{\dim \g{m}_{\pi/2}}$  \\		
\hline
\begin{minipage}{0.8cm}
\begin{enumerate}[{\rm (i)}]
\setcounter{enumi}{4}
\item \label{MainTheorem:5}
\end{enumerate}
\end{minipage} & $\g{s} = \R(B + U) \oplus  \g{m}_{\pi/2} \oplus \g{g}_{2\alpha}$ & \tabitem $\C U \perp \g{m}_{\pi/2}$ \qquad \qquad \tabitem  $\g{m}_{\pi/2} \neq 0$ & No & Solvable extension of $\R^{1+\dim \g{m}_{\pi/2}}$  \\
\hline
\multirow{2}{0.01cm}{\begin{minipage}{0.8cm}
\begin{enumerate}[{\rm (i)}]
\setcounter{enumi}{5}
\item \label{MainTheorem:6}
\end{enumerate}
\end{minipage}} & \multirow{2}{*}{$\g{s} = \R(B + U) \oplus \g{m}_{\varphi} \oplus  \g{m}_{\pi/2} \oplus \g{g}_{2\alpha}$} &  \tabitem $\g{m}_{\pi/2} \neq 0$ \qquad \qquad \tabitem $\C U \perp \g{m}_{\varphi} \oplus  \g{m}_{\pi/2}$ & \multirow{2}{*}{No} & \multirow{2}{4.5cm}{Solvable extension of a Heisenberg group $\times \R^{\dim \g{m}_{\pi/2}}$}  \\
& & \tabitem  $|U|^2 = \frac{\dim \g{m}_\varphi +\dim \g{m}_{\pi/2} + 4}{(\dim \g{m}_\varphi +4) \cos ^2(\varphi)}-1$ & & \\
\hline
\end{tabular}}
\label{table:MainTheorem}
\bigskip
\caption{List of the Lie algebras corresponding to the Lie subgroups of $AN \cong \C H^n$ that are Ricci solitons with the induced metric.}
\end{table}
\end{MainTheorem}
\begin{remark}[\textit{Notation for Table~\ref{table:MainTheorem} in the Main Theorem.}]\label{remark:notation}
In Table~\ref{table:MainTheorem} and in the rest of the work, the symbol $\oplus$ stands for orthogonal direct sum. The elements $U$, $V$ are vectors in $\g{g}_\alpha$, and $x$, $t$ are real numbers. For any $W \in \g{g}_\alpha$ and any vector subspace $\g{m}$ of $\g{a} \oplus \g{n}$, the notation $\C W \perp \g{m}$ means that both $W$ and $JW$ are orthogonal to $\g{m}$. The notation $\g{m}_\varphi$ stands for non-zero real subspaces of $\g{g}_\alpha$ of constant constant Kähler angle $\varphi \in [0, \pi/2)$, and $\g{m}_{\pi/2}$ for totally real subspaces of $\g{g}_\alpha$. Recall also that $\g{a} = \R B$, $\g{g}_{2\alpha} = \R Z$ and $JB = Z$.
\end{remark}

In the following lines, we detail the reason why we focus on the study of the Lie subgroups of the solvable Iwasawa group associated with a symmetric space of non-compact type (in this case a complex hyperbolic space). On the one hand, steady and shrinking homogeneous Ricci solitons are called trivial in the literature: the first are products of Euclidean spaces times a flat torus~\cite{AlKi}, and the investigation of the second can be reduced to that of compact Einstein manifolds~\cite{Iv93, Na10, PeWy09}. On the other hand, combining the recently solved \emph{(generalized) Alekseevskii conjecture}~\cite{BoLa21} and a remarkable result by Jablonski~\cite{Jab:arxiv} at the intersection of Ado's Theorem for Lie algebras and Nash embedding Theorem, one gets: \emph{any connected homogeneous expanding Ricci soliton is isometric to a Lie subgroup of the solvable part $AN$ of the Iwasawa decomposition associated with a symmetric space of non-compact type that is an algebraic Ricci soliton with the induced metric}. A Lie group $S$ endowed with a left-invariant metric is said to be an \emph{algebraic Ricci soliton} if its $(1, 1)$-Ricci tensor, $\Ric$, can be written as $\Ric = \bar{c} \id + D$, where $\bar{c}$ is a real number, and $D$ is a derivation of the Lie algebra $\g{s}$ of $S$, that is, an endomorphism of $\g{s}$ satisfying 
\begin{equation}\label{definition:derivation}
D[X,Y]=[D X,Y]+[X,D Y] \quad \text{for any} \quad X, Y \in \g{s}.
\end{equation} 
A solvable (respectively, nilpotent) algebraic Ricci soliton is called \emph{solvsoliton} (respectively, \emph{nilsoliton}). The aforementioned fact constitutes an important incentive towards the inspection of the Lie subgroups of solvable Iwasawa groups associated with symmetric spaces of non-compact type in order to find examples of Ricci solitons. This viewpoint was exploited in~\cite[Theorem~A]{DST20}, where the codimension one Lie subgroups of solvable Iwasawa groups associated with symmetric spaces of non-compact type that are Ricci solitons with the induced metric were classified. The only examples corresponding to $\C H^n$ are the generalized Heisenberg group $N$ and the Lohnherr hypersurface when $n = 2$. Our Main Theorem can be regarded as a natural continuation of the research line started in~\cite[Theorem~A]{DST20}, but removing the hypothesis on the codimension of the examples, and focusing on the solvable Iwasawa groups associated with a particular class of symmetric spaces, namely complex hyperbolic~spaces.

Since the solvable Iwasawa group associated with real hyperbolic spaces can be regarded as a Lie subgroup, with the induced metric, of the one corresponding to complex hyperbolic spaces, our Main Theorem contains in the first two items (taking $t = 0$) the analogous classification for real hyperbolic spaces. This can be regarded as some sort of partial generalization of Fialkow's work~\cite{Fi38} to any codimension (in the case of negative constant sectional curvature). Moreover, we have
\begin{corollary}\label{corollary:real:hyperbolic:space}
Let $S$ be a connected Lie subgroup of the solvable part $AN$ of the Iwasawa decomposition associated with the real hyperbolic space $\R H^n$. Assume that $S$ is a Ricci soliton with the induced metric. Then $S$ is Einstein and isometric to a Euclidean space or to a real hyperbolic space.
\end{corollary}

From the Main Theorem, we get the following characterization of the Einstein examples.

\begin{corollary}\label{corollary:totally:geodesic:symmetric:space}
Let $S$ be a connected Lie subgroup of the solvable part $AN$ of the Iwasawa decomposition associated with the complex hyperbolic space $\C H^n$, considered with the induced metric. Then, $S$ is Einstein if and only if it is a symmetric space.
\end{corollary}

In the aforementioned result by Jablonski~\cite{Jab:arxiv}, he provides an isometric embedding of certain homogeneous Ricci solitons in a symmetric space of non-compact type and he also poses the question of whether such embedding is optimal from, for example, the dimension viewpoint. From the submanifold geometry perspective, it is intriguing for us if this embedding could be achieved with the property of being minimal. In the proof of the following corollary, we will see that most of the examples of the Main Theorem are not minimal. However, minimality brings Hopf hypersurfaces, cohomogeneity one actions and isoparametric hypersurfaces into the picture. 

\begin{corollary}\label{corollary:hopf:isoparametric}
Let $S$ be a proper connected Lie subgroup of dimension $\geq 2$ of the solvable part $AN$ of the Iwasawa decomposition associated with the complex hyperbolic space $\C H^n$. Assume that $S$ is a minimal submanifold of $AN$ and a Ricci soliton when considered with the induced metric. Under the identification $\C H^n \cong AN$, one of the following conditions holds:
\begin{enumerate}[{\rm(i)}]
\item $S$ is a totally geodesic $\R H^k$ in $\C H^n$, for some $k \leq n$, and thus Einstein.
\item $S$ is a totally geodesic $\C H^k$ in $\C H^n$, for some $k \leq n$, and thus Einstein. 
\item $S$ is neither totally geodesic nor Einstein. Moreover: if $n > 2$, $S$ is the focal set of an isoparametric family of non-Hopf hypersurfaces that are homogeneous if and only if $S$ has totally real normal bundle; if $n = 2$, $S$ is the so-called Lohnherr hypersurface.
\end{enumerate}
In particular, $S$ is the focal set of an isoparametric family of hypersurfaces on a totally geodesic $\C H^k$ in $\C H^n$, for some $k \leq n$. However, not any focal set of an isoparametric family of hypersurfaces is a Ricci soliton.
\end{corollary}

It is well-known that minimal homogeneous submanifolds in real hyperbolic spaces are totally geodesic~\cite{DiOl}. This is no longer true in complex hyperbolic spaces, but as a consequence of Corollary~\ref{corollary:hopf:isoparametric} and the Main Theorem, we have

\begin{corollary}\label{corollary:einstein:minimal:totallygeodesic}
Let $S$ be a proper connected Lie subgroup of dimension $\geq 2$ of the solvable part $AN$ of the Iwasawa decomposition associated with the complex hyperbolic space $\C H^n$. Then:
\begin{enumerate}[{\rm(i)}]
\item $S$ is Einstein with the induced metric and a minimal submanifold of $AN$ if and only if it is totally geodesic.
\item $S$ is non-flat Einstein if and only if it is homothetic to a totally geodesic submanifold of $\C H^n$. 
\end{enumerate}
\end{corollary}

The strategy to prove the Main Theorem is the following. First, $AN$ is a completely solvable Lie algebra, and so are its Lie subalgebras. Hence, being a Ricci soliton is equivalent to being an algebraic Ricci soliton~\cite[p.~4]{La11C}, and the latter is the condition that we will actually consider. Moreover, Lauret proved that any solvsoliton $S$ can be recovered from its nilradical $L$, which is a nilsoliton with the induced metric, by means of a simple procedure that we make precise in Proposition~\ref{proposition:lauret} (\cite[Theorem~4.8]{La11C}). Therefore, we first classify the nilpotent Lie subgroups of $AN$ that are Ricci solitons with the induced metric in Proposition~\ref{proposition:nilradical}. After that, in Proposition~\ref{proposition:solvsolitons}, we analyze their possible extension to solvsolitons, in the sense of Proposition~\ref{proposition:lauret}, that can be still realized as Lie subgroups of $AN$. Indeed, any solvsoliton in the Main Theorem can be obtained as an extension, in the sense of Proposition~\ref{proposition:lauret}, of one of the nilsolitons in items~(\ref{MainTheorem:1}) or~(\ref{MainTheorem:4}) of the Main Theorem. In~\cite{S22}, it is stated that the nilradical of any solvsoliton can be isometrically embedded in the nilpotent Lie group $N$ of the Iwasawa decomposition associated with a symmetric space of non-compact type. Also in~\cite{S22}, the codimension one Ricci soliton Lie subgroups of any nilpotent Iwasawa group are classified, and they turn out to be minimal in $N$. In this line, and despite the existence of many non-minimal examples in the Main Theorem (see Corollary~\ref{corollary:hopf:isoparametric} and its proof), we have this

\begin{corollary}\label{corollary:minimal:n}
Let $S$ be a proper connected Lie subgroup of the solvable part $AN$ of the Iwasawa decomposition associated with the complex hyperbolic space $\C H^n$. Assume that $S$ is a Ricci soliton with the induced metric. If its nilradical $L$ is non-flat with the induced metric, then $L$ is a minimal submanifold of~$N$.
\end{corollary}

Let $L$ be a nilpotent Lie group endowed with a left-invariant metric. Then, $L$ is a nilsoliton if and only if it admits a rank one solvable extension $S$ ($\dim S = \dim L + 1$) that is a non-flat Einstein solvmanifold~\cite[Theorem~3.7]{La01}. Furthermore, a classification of Einstein solvmanifolds would follow from these rank one Einstein extensions of nilsolitons~\cite{He98}. So, let $L$ be now a Lie subgroup of the nilpotent part of the Iwasawa decomposition associated with a symmetric space of non-compact type $M$, and assume that it is a Ricci soliton with the induced metric. A natural question is if this rank one Einstein solvable extension of $L$ can still be regarded as a submanifold of $M$. This is almost never the case when $L$ has codimension one in $N$~\cite{S22}. From the Main Theorem we deduce that not any nilsoliton $L$ in $AN$ can be extended to an Einstein Lie subgroup of $AN$. However, we still have the following
\begin{corollary}\label{corollary:extension}
Let $L$ be a connected Lie subgroup of the nilpotent part $N$ of the Iwasawa decomposition associated with the complex hyperbolic space $\C H^n$. Assume that $L$ is a Ricci soliton. Then, $L$ can be always extended, in the sense of~\cite[Proposition~4.3]{La11C}, to a non-nilpotent Lie subgroup of $AN$ that is a Ricci soliton with the induced metric. 
\end{corollary}

This paper is organized as follows. In Section~\ref{section:preliminaries} we introduce some tools and notations concerning mainly complex hyperbolic spaces and submanifold geometry. After that, Section~\ref{section:nilradical} is devoted to classifying the nilpotent Lie subgroups of $AN$ which are Ricci solitons with the induced metric. More precisely, since for nilpotent Lie groups endowed with a left-invariant metric, being a Ricci soliton is equivalent to being an algebraic Ricci soliton, what we actually do is classifying the nilpotent Lie subgroups of $AN$ that are algebraic Ricci solitons (see Proposition~\ref{proposition:nilradical}). Finally, in Section~\ref{section:solvsolitons}, we analyze which extensions of these nilsolitons, in the sense of Proposition~\ref{proposition:lauret}, can be still achieved as submanifolds of complex hyperbolic spaces (Proposition~\ref{proposition:solvsolitons}), and then we prove the Main Theorem along with its corollaries.  

\bigskip

We would like to thank Miguel Domínguez-Vázquez and Jorge Lauret for their helpful comments on several aspects of this work.

\section{Preliminaries} \label{section:preliminaries}
This section is devoted to the introduction of the terminology and the basic notions concerning mainly complex hyperbolic spaces, as well as some necessary tools to address the analysis of a certain class of their submanifolds.

In this line, we denote by $\C H^n$ the \emph{complex hyperbolic space} of real dimension $2n$ and constant holomorphic sectional curvature $-1$, which is a Hadamard manifold diffeomorphic to $\R^{2n}$, and a Kähler manifold whose \emph{complex structure} will be denoted by $J$. The complex hyperbolic space $\C H^n$ is also a symmetric space of non-compact type, and it can be expressed as the quotient of the Lie groups $G/K$, where $G = SU(1,n)$ and $K=S(U(1)\times U(n))$. Along this work, Lie algebras will be denoted by gothic letters. Thus, let $\g{g} = \g{k} \oplus \g{p}$ be the \emph{Cartan decomposition} of the Lie algebra $\g{g}$ of $G$, where $\g{k}$ is the Lie algebra of $K$, and $\g{p}$ is the orthogonal complement of $\g{k}$ in $\g{g}$ with respect to the Killing form $\mathcal{B}$ of $\g{g}$, which is non-degenerate since $G$ is a real simple group~\cite[Theorem~1.45]{K}. Let $\theta$ denote the corresponding Cartan involution. Now, we define an inner product $\langle \cdot, \cdot \rangle_{\mathcal{B}_\theta}$ on $\g{g}$ by the expression $\langle X, Y \rangle_{\mathcal{B}_\theta} = -\mathcal{B}(\theta X, Y)$, for any $X$, $Y \in \g{g}$. Take $\g{a}$ a maximal abelian subspace of $\g{p}$, which turns out to be one-dimensional. This implies that $\C H^n$ is a rank one symmetric space, which is equivalent to say that any of its maximal flat totally geodesic submanifolds have dimension one. The endomorphisms of $\g{g}$ in the family $\{ \ad(H): H \in \g{a} \}$ are self-adjoint with respect to $\langle \cdot, \cdot \rangle_{\mathcal{B}_\theta}$ and they commute with each other. This implies that they diagonalize simultaneously giving rise to the so-called \emph{(restricted) root space decomposition} $\g{g} = \g{g}_{-2 \alpha} \oplus \g{g}_{-\alpha} \oplus \g{g}_0 \oplus \g{g}_{\alpha} \oplus \g{g}_{2 \alpha}$, where each of the \emph{(restricted) root spaces} is defined as $\g{g}_\lambda = \{ X \in \g{g} \, : \, \ad(H) X = \lambda(H) X, \, \text{for all }  H \in \g{a}\}$, with $\lambda$ in $\g{a}^*$, that is, $\lambda$ in the dual space of $\g{a}$.

The \emph{Iwasawa decomposition theorem}, at the Lie algebra level, states that $\g{g} = \g{k} \oplus \g{a} \oplus \g{n}$, where $\g{n} = \g{g}_{\alpha} \oplus \g{g}_{2 \alpha}$~\cite[Theorem~6.43]{K}. The set $\g{n}$ is a nilpotent Lie subalgebra of $\g{g}$, as $[\g{g}_\lambda, \g{g}_\beta] \subset  \g{g}_{\lambda+\beta}$ for any $\lambda$, $\beta \in \g{a}^*$. The connected solvable Lie subgroup of $G$ whose Lie algebra is $\g{a} \oplus \g{n}$, which we denote by $AN$ in what follows, turns out to be diffeomorphic to $\C H^n$ and thus to $\R^{2n}$. Indeed, the Lie exponential map $\Exp \colon \g{a} \oplus \g{n} \to AN$ is a diffeomorphism, which means that any connected Lie subgroup $S$ of $AN$ is diffeomorphic to some Euclidean space. Now, we consider in $AN$ the metric $\langle \cdot, \cdot \rangle$ making it and $\C H^n$ isometric manifolds, which happens to be left-invariant. Furthermore, $J$ turns out to be an \emph{orthogonal complex structure} in $\g{a}\oplus \g{n}$ which leaves $\g{g}_\alpha$ invariant and satisfies $J \g{a} = \g{g}_{2\alpha}$. From now on, $B$ and $Z$ will be unit vectors with respect to the inner product $\langle \cdot, \cdot \rangle$ in $\g{a}$ and $\g{g}_{2 \alpha}$, respectively, satisfying $JB = Z$. Thus,  $\g{a} = \R B$ and $\g{g}_{2\alpha} = \R Z$. Indeed, whenever we specify the length of a vector, we will understand it with respect to the inner product  $\langle \cdot, \cdot \rangle$. Just for the sake of completeness, the relation between the two inner products that we have in $\g{a} \oplus \g{n}$ is given by $\langle B, B \rangle = \langle B, B \rangle_{\mathcal{B}_\theta}$ and $2 \langle X, Y \rangle = \langle X, Y \rangle_{\mathcal{B}_\theta}$, for any $X$, $Y \in \g{n}$. For any $U$, $V \in \g{g}_\alpha$, the Lie bracket in $AN$ is given by the relations
\begin{equation}\label{equation:brackets}
\begin{aligned}
\relax [B,U]=\frac{1}{2}U, \quad
[U,V]=\langle JU,V \rangle Z, \quad
[B,Z]=Z, \quad
[Z,U]=0.
\end{aligned}
\end{equation}
The Levi-Civita connection of $AN$ reads as~(\cite[Section~4.1.6]{BTV95})
\begin{equation}\label{equation:levi-civita}
\begin{aligned}
\nabla_{aB +U + cZ}(bB + V +dZ)= & \phantom{+} \left( \frac{1}{2} \langle U, V\rangle +cd\right)B - \frac{1}{2}(bU+cJV + dJU) \\ & +\left(\frac{1}{2}\langle JU, V \rangle -bc\right) Z,
\end{aligned}
\end{equation}
for any $U$, $V \in \g{g}_\alpha$ and any real numbers $a$, $b$, $c$ and $d$. We include below a result in order to use the Levi-Civita connection of $AN$ more efficiently, which follows directly from~\eqref{equation:levi-civita} and the fact that $J$ is an orthogonal complex structure of~$\g{g}_\alpha$.
\begin{lemma}\label{lemma:nablas}
Let $\xi$, $X$ be orthogonal vectors in $\g{g}_\alpha$. Then:
\begin{multicols}{3}
\begin{enumerate}[{\rm (i)}]
\item $\nabla_X \xi = -\frac{1}{2}  \langle J \xi, X \rangle Z$,\label{lemma:nablas:1}
\item $\nabla_X X =\frac{1}{2} \langle  X, X \rangle B$, \label{lemma:nablas:2}
\item $\nabla_X B = -\frac{1}{2} X$, \label{lemma:nablas:3}
\item $\nabla_Z B = -Z$, \label{lemma:nablas:4}
\item $\nabla_\xi Z = \nabla_Z \xi =  -\frac{1}{2} J \xi$.\label{lemma:nablas:5}
\end{enumerate}
\end{multicols}
\end{lemma}
The $(1,1)$-Ricci tensor of $AN$ reads as $\Ric^{AN} = -((n+1)/2) \id$~(\cite[Section~4.1.7]{BTV95}), proving in particular that complex hyperbolic spaces are isometric to Einstein solvmanifolds.

Let $V$ be a vector space endowed with an inner product over $\R$ or over $\C$. If $W$ is a subspace of $V$, then $V \ominus W$ will denote the orthogonal complement of $W$ in $V$ with respect to this inner product. Moreover, if $X$ is a subset or a vector of $V$, we denote by $\R X$ and $\C X$ (when it makes sense) the real and the complex span of $X$ in $V$, respectively.

Since we will be working with the Lie algebra $\g{g} = \g{su}(1,n)$ of $SU(1,n)$, and $\g{g}_\alpha$ turns out to be isomorphic to $\C^{n-1}$, it will be important to understand the behavior of real subspaces of complex vector spaces. In order to address this question, we recall the notion of \emph{Kähler angle}. Consider the complex Euclidean space $\C^n$ and denote by $J$ its complex structure. Let $V$ be a real vector subspace of $\C^n$,  that is, a real vector subspace of $\C^n$, when we think of $\C^n$ as a real vector space by restricting the scalars to the real numbers. The Kähler angle of a non-zero vector $v \in V$ with respect to $V$ is defined to be the angle between $J v$ and $V$. We will say that $V$ has \emph{constant Kähler angle} $\varphi \in [0, \pi/2]$ if every non-zero vector $v \in V$ has Kähler angle $\varphi$ with respect to $V$. In particular, this means that $V$ is a complex subspace of $\C^n$ if and only if it has constant Kähler angle $0$. In addition, $V$ is said to be a \emph{totally real} subspace of $\C^n$ if it has constant Kähler angle $\pi/2$. The following result, which follows from combining~\cite[Theorem~2.7 and Lemma~2.8]{DDK17} with~\cite[Proposition~7]{BB01}, constitutes an important tool in order to approach the proof of the Main~Theorem. 

\begin{proposition}\label{proposition:kahler:angle}
Let $V$ be a real subspace of $\C^n$. Then, $V$ can be decomposed in a unique way as an orthogonal sum of subspaces $V_\varphi$ of constant Kähler angle $\varphi$, with $\varphi$ in a finite subset $\Phi$ of $[0, \pi/2]$. Moreover, we have:
\begin{enumerate}[{\rm (i)}]
\item $\C V_\varphi$ is orthogonal to $\C V_\psi$ for any distinct $\varphi$, $\psi \in \Phi$. \label{proposition:kahler:angle:1}
\item $V_\varphi^\perp = \C V_\varphi \ominus  V_\varphi$ has constant Kähler angle $\varphi$ and the same real dimension as $V_\varphi$, for any $\varphi \in \Phi \backslash \{ 0 \}$.\label{proposition:kahler:angle:2}
\item For any $\varphi \in \Phi \cap (0, \pi/2)$, put $\dim V_\varphi = 2 k_\varphi$. Then, there exist $2k_\varphi$ $\C$-orthonormal vectors $e_1, \dots, e_{2k_\varphi}$ in $\C^n$ inducing the orthogonal decomposition $V_\varphi = \bigoplus_{i = 1}^{k_\varphi} V_\varphi^i$, where $V_\varphi^l = \spann \{ e_{2l-1}, \cos \varphi J e_{2l-1} + \sin \varphi J e_{2l} \}$ for each $l \in \{1, \dots, k_\varphi\}$. Furthermore, $\C V_\varphi^l$ is orthogonal to $\C V_\varphi^m$ for any distinct $l$, $m \in \{1, \dots, k_\varphi\}$. \label{proposition:kahler:angle:3}
\end{enumerate}
\end{proposition}

In order to finish this section, we introduce some necessary tools to address the analysis of the geometry of the Lie subgroups of $AN$. Let $S$ be a connected Lie subgroup of $AN$. For our purposes, it will suffice to study its geometry at the neutral element, where we can identify the tangent space to $S$ with its Lie algebra $\g{s}$ and its normal space with $(\g{a} \oplus \g{n}) \ominus \g{s}$. Take a normal vector $\xi \in (\g{a} \oplus \g{n}) \ominus \g{s}$. On the one hand, $\Ss_\xi$ denotes the \emph{shape operator} of $S$ with respect to $\xi$, that is, the endomorphism of $\g{s}$ defined as $\Ss_{\xi} X = - (\nabla_X \xi)^\top$, for any $X \in \g{s}$, where $(\cdot)^\top$ denotes the orthogonal projection onto the tangent space $\g{s}$ of $S$. This allows to define the second fundamental form $\II$ of $S$ by the relation $\langle \II(X,Y), \xi \rangle = \langle \Ss_\xi X, Y \rangle $ for any $X$, $Y \in \g{s}$, and any $\xi \in (\g{a} \oplus \g{n}) \ominus \g{s}$. On the other hand, $R_\xi$ denotes the \emph{Jacobi operator}, which is defined by $R_\xi X = R(X, \xi) \xi$ for any $X \in \g{s}$, where $R$ is the curvature tensor of $AN$, for which we use the sign convention $R (X,Y)Z = \nabla_X \nabla_Y Z -\nabla_Y \nabla_X Z - \nabla_{[X,Y]} Z$, for any $X$, $Y$, $Z \in \g{a} \oplus \g{n}$. For the sake of completeness, we include below the Gauss equation for tangent vectors $X$, $Y$, $Z$, $W$ to $S$, where $R^{AN}$ and $R^S$ denote the curvature tensors of $AN$ and of $S$, respectively. It reads as
\begin{equation}\label{equation:Gauss}
\langle R^{AN}(X,Y)Z,W\rangle =\langle R^S(X,Y)Z,W\rangle-\langle\II(Y,Z),\II(X,W)\rangle + \langle\II(X,Z),\II(Y,W)\rangle.
\end{equation}

\section{The nilradical}\label{section:nilradical}
Any solvsoliton can be obtained, via a surprisingly simple procedure, from its nilradical, which is known to be a nilsoliton. This is detailed later in Proposition~\ref{proposition:lauret} (see~\cite[Theorem~4.8]{La11C}). This fact suggests addressing the proof of the Main Theorem by inspecting first the nilradicals of solvsolitons, that is, nilsolitons. In this line, we start this section by proving that the nilradical of any Lie subgroup of $AN$ of dimension $\geq 2$ has to be contained in the nilpotent Lie group $N$. After that, the main objective of the section is classifying all the connected Lie subgroups of $N$ that are Ricci solitons with the induced metric, which is equivalent to classifying the nilradicals of the connected Lie subgroups of $AN$ that are Ricci solitons with the induced metric.

Hence, let $S$ be a connected Lie subgroup of the solvable part $AN$ of the Iwasawa decomposition of the connected component of the identity of the isometry group of a complex hyperbolic space. Recall from Section~\ref{section:preliminaries} that $S$ is diffeomorphic to some Euclidean space. Then, if $S$ is one-dimensional, it is flat and thus isometric to a Euclidean space.  Therefore, from now on we will consider Lie subgroups of $AN$ of dimension at least~two. We start with an auxiliary result.

\begin{lemma}\label{lemma:nilpotent:nilradical}
Let $\g{s}$ be a Lie subalgebra of $\g{a} \oplus \g{n}$ of dimension $\geq 2$. Then:
\begin{enumerate}[{\rm (i)}]
\item The nilradical $\g{l}$ of $\g{s}$ is contained in $\g{n}$. \label{lemma:nilpotent:nilradical:1}
\item The subalgebra $\g{s}$ is non-nilpotent if and only if it can be written as $\g{s} = \R (B + U + x Z) \oplus \g{l}$, for some $U \in \g{g}_\alpha$, some $x \in \R$, and where $\g{l}$ is the nilradical of $\g{s}$. \label{lemma:nilpotent:nilradical:2}
\end{enumerate}
\end{lemma}

\begin{proof}
We will first prove~(\ref{lemma:nilpotent:nilradical:2}) and then derive~(\ref{lemma:nilpotent:nilradical:1}) from there. Therefore, we will assume that $\g{s}$ is non-nilpotent and thus not contained in~$\g{n}$. Without loss of generality, we can and will consider the orthogonal decomposition $\g{s} = \R (B+U+xZ) \oplus \g{m}$, for some non-zero subspace $\g{m}$ of $\g{n}$, some $U \in \g{g}_\alpha$ and some $x \in \R$. Since $[\g{a}\oplus \g{n}, \g{a} \oplus \g{n}] \subset \g{n}$, $\g{s}$ is a subalgebra and $B + U + x Z$ has non-trivial projection onto $\g{a}$, then $[\g{s}, \g{m}] \subset \g{n} \cap \g{s} = \g{m}$. Thus, $\g{m}$ is a nilpotent ideal of codimension one in $\g{s}$. Since $\g{s}$ is non-nilpotent by assumption, then $\g{m} \subset \g{n}$ is the nilradical $\g{l}$ of $\g{s}$.

For the converse, define $\g{h} = \R (B+U+xZ) \oplus \g{m}$, for some non-zero subspace $\g{m}$ of $\g{n}$, some $U \in \g{g}_\alpha$ and some $x \in \R$. As follows from~\eqref{equation:brackets}, $\g{h}$ is a subalgebra of~$\g{a} \oplus \g{n}$. Take a non-zero vector $ V + y Z$ in $\g{m} \subset \g{h}$, with $V \in \g{g}_\alpha$ and $y \in \R$. Using~\eqref{equation:brackets} and induction hypothesis, we get 
\[
\ad^{n}(B + U + x Z )(V+y Z)= \frac{1}{2^{n}}V+\left(y+\sum_{k=0}^{n-1}\frac{1}{2^{k}}\langle JU, V \rangle\right)Z,
\]
which is a non-zero element for each positive integer $n$. Indeed, if $V \neq 0$ we are done. Otherwise, $V = 0$ and thus $\langle JU, V \rangle = 0$, but $y \neq 0$ as $V+yZ$ was taken non-trivial. This proves that the lower central series of $\g{h}$ is never zero, which implies that $\g{h}$ is non-nilpotent. Since $\g{s}$ in~(\ref{lemma:nilpotent:nilradical:2}) is under the conditions of $\g{h}$, then~(\ref{lemma:nilpotent:nilradical:2}) follows.

If $\g{s}$ is contained in $\g{n}$, then~(\ref{lemma:nilpotent:nilradical:1}) follows directly. Otherwise, it can be written as $\g{h}$ above, it is non-nilpotent and from~(\ref{lemma:nilpotent:nilradical:2}) its nilradical is in~$\g{n}$. 
\end{proof}

\begin{remark}
An analogous result to Lemma~\ref{lemma:nilpotent:nilradical} is still true if $\g{a} \oplus \g{n}$ is a rank one metric Lie algebra of Iwasawa type~\cite[Definition~1.2]{Wolter}. More precisely, any subalgebra $\g{s}$ of dimension at least two of $\g{a} \oplus \g{n}$ can be written as $\R(B+X) \oplus \g{l}$ for some $X \in \g{n}$, where $\g{a} = \R B$ and $\g{l} \subset \g{n}$ is the nilradical of $\g{s}$. The key point is that $\ad (B)$ is positive definite in $\g{n}$.
\end{remark}

As mentioned above and detailed later in Proposition~\ref{proposition:lauret}, solvsolitons can be constructed from their nilradicals, which are nilsolitons. Thus, taking into account Lemma~\ref{lemma:nilpotent:nilradical}~(\ref{lemma:nilpotent:nilradical:1}), we focus our attention on the connected Lie subgroups of $N$ that are Ricci solitons when considered with the induced metric. Let $L$ be a connected Lie subgroup of $N$ with Lie algebra $\g{l}$. If $\g{g}_{2\alpha}$ is not contained in $\g{l}$, then $\g{l}$ is abelian by~\eqref{equation:brackets} and this case will be addressed directly in Proposition~\ref{proposition:nilradical}~(\ref{proposition:nilradical:1}). Then, put $\g{l} = \g{m} \oplus \g{g}_{2 \alpha}$, where $\g{m}$ is a real subspace of $\g{g}_\alpha$. Recall that $\Ss_\xi$ and $R_\xi$ denote the shape operator and the Jacobi operator of $L$ as a submanifold of $AN$, respectively, with respect to a normal vector~$\xi$. Let $\mathcal{H} = \mathcal{H}_\g{a} + \mathcal{H}_\g{n}$ be the mean curvature vector of $L$ as a submanifold of $AN$, where $(\cdot)_{\g{a}}$ and $(\cdot)_{\g{n}}$ denote the orthogonal projection onto $\g{a}$ and onto $\g{n}$, respectively. Since $\g{a} = \R B$, we deduce $\mathcal{H}_{\g{a}} = (\tr \Ss_B) B$. From now on, the notation~$(\cdot)^\top$ will denote orthogonal projection onto $\g{l}$. From~\eqref{equation:levi-civita}, we easily deduce that $\nabla_ B X = 0$ for any $X \in \g{a} \oplus \g{n}$. Using this, \eqref{equation:brackets} and Lemma~\ref{lemma:nablas}~(\ref{lemma:nablas:3})-(\ref{lemma:nablas:4}), we obtain directly
\begin{lemma}\label{lemma:ricN}
Let $X$ be in $\g{l} = \g{m} \oplus \g{g}_{2 \alpha}$, with $\g{m}$ a real subspace of $\g{g}_\alpha$. Then, we have:
\begin{multicols}{2}
\begin{enumerate}[{\rm (i)}]
\item $(R^\top_B + \Ss^2_{B}) X = 0$.\label{lemma:ricN:3}
\item $\Ss_{\mathcal{H}_\g{a}} X =(\tr \Ss_B) \ad(B) X$.\label{lemma:ricN:4}
\end{enumerate}
\end{multicols}
\end{lemma}

Let $\{\xi_i \, : \, i = 1, \dots, k\}$ be an orthonormal basis of the orthogonal complement of $\g{l}$ in $\g{n}$. Using this, recalling $\Ric^{AN} = -((n+1)/2) \id$ from Section~\ref{section:preliminaries}, and combining the Gauss equation~\eqref{equation:Gauss} with Lemma~\ref{lemma:ricN}, one can derive the following expression for the $(1, 1)$-Ricci tensor of $L$:
\begin{equation}\label{equation:ricci:s}
\Ric^{L} =  -\frac{(n+1)}{2} \id + (\tr \Ss_B) \ad(B) + \Ss_{\mathcal{H}_{\g{n}}} - \sum_{i = 1}^k (\Ss_{\xi_i}^2 + R^\top_{\xi_i}).
\end{equation}

Let us start with the inspection of the terms involved in~\eqref{equation:ricci:s}. Recall that $\g{l} = \g{m} \oplus \g{g}_{2\alpha}$, where $\g{m}$ is a real subspace of $\g{g}_\alpha$ and $\g{g}_{2 \alpha} = \R Z$.
\begin{lemma}\label{lemma:shape:xi}
Let $X$ be a vector in $\g{m}$ and $\xi$ a unit vector in $\g{g}_\alpha \ominus \g{m}$. Then, we have:
\begin{multicols}{2}
\begin{enumerate}[{\rm (i)}]
\item $\Ss_\xi X = \frac{1}{2}  \langle J \xi, X \rangle Z $, \label{lemma:shape:xi:1}
\item $\Ss_\xi Z = \frac{1}{2}  J\xi^{\top} $, \label{lemma:shape:xi:2}
\item $\Ss^2_\xi X =  \frac{1}{4} \langle J \xi, X \rangle J\xi^\top$,\label{lemma:shape:xi:3}
\item$\Ss^2_\xi Z =  \frac{1}{4} \langle J \xi^\top, J \xi^\top \rangle  Z$.\label{lemma:shape:xi:4}
\end{enumerate}
\end{multicols}
Moreover, $\tr \Ss_\xi = 0$ for any $\xi \in \g{g}_\alpha \ominus \g{m}$.
\end{lemma}

\begin{proof}
The statements from~(\ref{lemma:shape:xi:1}) to~(\ref{lemma:shape:xi:4}) follow directly from using Lemma~\ref{lemma:nablas} and the definition of the shape operator, that is, $\Ss_\eta Y = -(\nabla_Y \eta)^\top$ for any $Y \in \g{l}$ and any $\eta \in \g{g}_\alpha \ominus \g{m}$.

Let us focus on the last claim of this result concerning the trace. From~(\ref{lemma:shape:xi:1}), we have $\langle \Ss_{\xi}X,X\rangle= (1/2) \langle J \xi, X \rangle \langle Z,X\rangle=0$ for any $X \in \g{m}$. Moreover, using~(\ref{lemma:shape:xi:2}) we deduce that $\langle \Ss_{\xi}Z,Z\rangle= (1/2) \langle J \xi^{\top},Z\rangle=0$, which follows since $J \xi$ is in $\g{g}_\alpha$ and thus $J\xi^\top$ belongs to $\g{m}$.  Since $\g{l} = \g{m} \oplus \g{g}_{2\alpha}$, this implies  that $\tr \Ss_\xi = 0$ for any $\xi \in \g{g}_\alpha \ominus \g{m}$. 
\end{proof}

Recall that $L$ denotes the connected Lie subgroup of $N$ with Lie algebra $\g{l} = \g{m} \oplus \g{g}_{2 \alpha}$, where $\g{m}$ is a real subspace of $\g{g}_\alpha$. Note that $\ad(B)^\top X = l X$, where $l$ is $1/2$ or $1$ depending on whether $X$ is in $\g{m}$ or in $\g{g}_{2 \alpha}$, respectively, by means of~\eqref{equation:brackets}. Then $(\ad(B)\rvert_{\g{l}})^\top = \ad(B) \rvert_{\g{l}}$ is trivially a derivation of $\g{l}$. Moreover, from the claim concerning traces in Lemma~\ref{lemma:shape:xi} we get that $\mathcal{H}_{\g{n}} =  \sum_{j = 1}^k (\tr \Ss_{\xi_j}) \xi_j = 0$ and thus~\eqref{equation:ricci:s} reads now as
\begin{equation}\label{equation:ricci:l}
\Ric^{L} =  -\frac{(n+1)}{2}  \id + (\tr \Ss_B) \ad(B) - \sum_{i = 1}^k (\Ss_{\xi_i}^2 + R^\top_{\xi_i}).
\end{equation}
Therefore, using~\eqref{equation:ricci:l}, we deduce that $\g{l}$ is an algebraic Ricci soliton if and only if, for some real value of $c$, the endomorphism $\mathcal{D}$ of $\g{l}$ defined as 
\begin{equation}\label{equation:derivation:nilsolitons}
\mathcal{D} = \sum_{i = 1}^k (\Ss_{\xi_i}^2 + R^\top_{\xi_i}) + c \id
\end{equation}
is also a derivation of $\g{l}$. Next step consists in calculating the remaining unknown terms involved in~\eqref{equation:derivation:nilsolitons}.
\begin{lemma}\label{lemma:Jacobi:xi}
Let $\xi$ and $X$ be unit vectors in $\g{g}_\alpha \ominus \g{m}$ and $\g{m}$, respectively. Then:
\begin{multicols}{2}
\begin{enumerate}[{\rm (i)}]
\item $R^\top_\xi X = -\frac{1}{4} X - \frac{3}{4}\langle J \xi, X \rangle J \xi^\top  $, \label{lemma:Jacobi:xi:1}
\item $R^\top_\xi Z = -\frac{1}{4} Z  $. \label{lemma:Jacobi:xi:2}
\end{enumerate}
\end{multicols}
\end{lemma}
\begin{proof}
(\ref{lemma:Jacobi:xi:1}): Using Lemma~\ref{lemma:nablas}~(\ref{lemma:nablas:2})-(\ref{lemma:nablas:1}) and \eqref{equation:brackets} in the second equality below, Lemma~\ref{lemma:nablas}~(\ref{lemma:nablas:3})-(\ref{lemma:nablas:5}) and that $J$ is an orthogonal complex structure in the third equality below, we get 
\begin{equation*}
\begin{aligned}
R_{\xi}X &=  \nabla_{X}\nabla_{\xi} \xi - \nabla_{\xi}\nabla_{X} \xi - \nabla_{[X,\xi]}\xi =\frac{1}{2}\nabla_{X}B+\frac{1}{2}\langle J\xi, X\rangle\nabla_{\xi} Z - \langle JX, \xi \rangle \nabla_{Z}\xi \\
&= -\frac{1}{4}X - \frac{1}{4}\langle J\xi, X\rangle J\xi - \frac{1}{2}\langle X, J\xi\rangle J\xi =  -\frac{1}{4} X - \frac{3}{4}\langle J \xi, X \rangle J \xi.
\end{aligned}
\end{equation*}
(\ref{lemma:Jacobi:xi:2}): Using Lemma~\ref{lemma:nablas}~(\ref{lemma:nablas:2})-(\ref{lemma:nablas:5}) and~\eqref{equation:brackets} in the second equality below, Lemma~\ref{lemma:nablas}~(\ref{lemma:nablas:4})-(\ref{lemma:nablas:1}) in the third equality below, and $J^{2}=-\id$ in the fourth equality below, we obtain 
\begin{equation*}
R_{\xi}Z = \nabla_{Z}\nabla_{\xi} \xi - \nabla_{\xi}\nabla_{Z} \xi - \nabla_{[Z,\xi]}\xi = \frac{1}{2}\nabla_{Z}B+\frac{1}{2}\nabla_{\xi}J\xi = -\frac{1}{2}Z-\frac{1}{4}\langle \xi, J^2 \xi\rangle Z =-\frac{1}{4}Z.\qedhere \\
\end{equation*}
\end{proof}

Recall that $L$ denotes the connected Lie subgroup of $N$ whose Lie algebra is $\g{l} = \g{m} \oplus \g{g}_{2 \alpha}$, where $\g{m}$ is a real subspace of $\g{g}_\alpha$. According to Proposition~\ref{proposition:kahler:angle}, for a certain finite subset $\Phi$ of $[0, \pi/2]$, we have the orthogonal decompositions
\begin{equation}\label{equation:decompositions:nilradical}
\g{m} = \bigoplus_{\psi \in \Phi} \g{m}_\psi \quad  \text{and}  \quad \g{m}_\varphi = \bigoplus_{i = 1}^{k_\varphi} V_\varphi^i \,, \, \quad \text{for any} \quad \varphi \in \Phi \cap (0, \pi/2).
\end{equation}
Now, we calculate the endomorphism $\mathcal{D}$ of $\g{l}$ defined in~\eqref{equation:derivation:nilsolitons} by making using of the above decompositions.
\begin{lemma}\label{lemma:D}
We have:
\begin{enumerate}[{\rm (i)}]
\item $\mathcal{D} \rvert_{\g{m}_\varphi} = \left(-\frac{1}{4} \dim (\g{g}_\alpha \ominus \g{m}) - \frac{1}{2} \sin^2\varphi +c \right) \id$ for any $\varphi \in \Phi$. \label{lemma:D:1}
\item $\mathcal{D} \rvert_{\g{g}_{2\alpha}} = \left(-\frac{1}{4} \dim (\g{g}_\alpha \ominus \g{m}) + \frac{1}{4} \sum_{\psi \in \Phi} \dim \g{m}_\psi \sin^2 \psi +c \right) \id$.\label{lemma:D:2}
\end{enumerate}
\end{lemma}

\begin{proof}
(\ref{lemma:D:1}): Take $X \in \g{m}$. Recall that $\{ \xi_i \, : \, i = 1, \dots, k \}$ is an orthonormal basis of $\g{g}_\alpha \ominus \g{m}$. Then, from definition~\eqref{equation:derivation:nilsolitons}, together with Lemma~\ref{lemma:shape:xi}~(\ref{lemma:shape:xi:3}) and Lemma~\ref{lemma:Jacobi:xi}~(\ref{lemma:Jacobi:xi:1}), we deduce
\begin{equation}\label{equation:d:mphi}
\mathcal{D} X = \left(-\frac{1}{4} \dim (\g{g}_\alpha \ominus \g{m}) + c\right) X - \frac{1}{2} \sum_{i = 1}^k \langle J \xi_i, X \rangle J \xi_i^\top. 
\end{equation}
On the one hand, if $X \in \g{m}_0$, then $JX \in \g{m}_0$, as $\g{m}_0$ is a complex subspace of $\g{g}_\alpha$. Hence $\langle  J\xi, X \rangle = -\langle \xi, JX \rangle = 0$ for any normal vector $\xi$ in $\g{g}_\alpha \ominus \g{m}$. Thus, the second term in the right hand side of~\eqref{equation:d:mphi} vanishes and we obtain~(\ref{lemma:D:1}) when $\varphi = 0$.

On the other hand, take $X \in \g{m}_{\pi/2}$. We can and will assume  $\langle X, X \rangle = 1$, as $\mathcal{D}$ is an endomorphism. Using Proposition~\ref{proposition:kahler:angle}~(\ref{proposition:kahler:angle:1}) and the fact that $\g{m}_{\pi/2}$ is totally real,  we get that $JX$ is a unit normal vector to $\g{l}$. Thus, we can assume that $JX$ is an element of the orthonormal basis $\{ \xi_i \, : \, i = 1, \dots, k \}$. Using this and that $J$ is an orthogonal complex structure of $\g{g}_\alpha$, we get~(\ref{lemma:D:1}) when~$\varphi = \pi/2$.

Finally, let us assume that $X \in V_\varphi^l \subset \g{m}_\varphi$, for some $\varphi \in \Phi \cap (0, \pi/2)$ and some $l \in \{1, \dots, k_\varphi\}$. For a certain complex subspace $\g{c}$ of $\g{g}_\alpha$, we can consider the orthogonal decomposition of the normal space to $\g{l}$ in~$\g{n}$ given by
\begin{equation}\label{equation:decomposition:normal}
\g{g}_\alpha \ominus \g{m} = \g{c} \oplus \left( \bigoplus_{\varphi \in \Phi \backslash \{0\}} \g{m}_\varphi^\perp  \right),
\end{equation}
where $\g{m}_\varphi^\perp = \C \g{m}_\varphi \ominus \g{m}_\varphi$ for each non-zero $\varphi \in \Phi$, as defined in Proposition~\ref{proposition:kahler:angle}~(\ref{proposition:kahler:angle:2}). Take $\xi \in \g{m}_\psi^\perp$, or $\xi \in \C V_\varphi^m \ominus V_\varphi^m$, or $\xi \in \g{c}$, for any $\psi \in \Phi \backslash \{0\}$ different from $\varphi$ and any $m \in \{1, \dots, k_\varphi\}$ different from $l$. Then $\langle J \xi, X \rangle = 0$. This follows from Proposition~\ref{proposition:kahler:angle}~(\ref{proposition:kahler:angle:1}) when $\xi \in \g{m}_\psi^\perp$, from Proposition~\ref{proposition:kahler:angle}~(\ref{proposition:kahler:angle:3}) when $\xi \in \C V_\varphi^m \ominus V_\varphi^m$, and from the fact that $\g{c}$ is a complex subspace of $\g{g}_\alpha \ominus \g{m}$ when $\xi \in \g{c}$. Thus, in order to calculate $\mathcal{D} X$ with $X \in V_\varphi^l$, it suffices to consider normal vectors to $\g{l}$ in $\C V_\varphi^l \ominus V_\varphi^l$.  

According to Proposition~\ref{proposition:kahler:angle}~(\ref{proposition:kahler:angle:3}), we can write $V_\varphi^l = \spann \{X, Y\}$ and $\C V_\varphi^l \ominus V_\varphi^l = \spann \{ \eta_1, \eta_2\}$, with $X = e_{2l-1}$, $Y = \cos \varphi J e_{2l-1} + \sin \varphi J e_{2l}$, $\eta_1 = e_{2l}$, and $\eta_2 = -\sin \varphi J e_{2l-1} + \cos \varphi J e_{2l}$, where ${e_{2l-1}, e_{2l}}$ are unit $\C$-orthogonal vectors. Moreover, note that $J \eta^\top$ is in $V_\varphi^l$ for any $\eta \in \C V_\varphi^l \ominus V_\varphi^l$, as follows from Proposition~\ref{proposition:kahler:angle}~(\ref{proposition:kahler:angle:3}). Thus, we get $J \eta_1^\top = \sin \varphi \, Y$ and $J \eta_2^\top = \sin \varphi \, X$. Using all these considerations and the expressions for $X$, $Y$, $\eta_1$ and $\eta_2$ given above, together with~\eqref{equation:d:mphi}, we obtain~(\ref{lemma:D:1}) for $\varphi \in \Phi \cap (0, \pi/2)$. 

(\ref{lemma:D:2}): Take an orthonormal basis $\{ \xi_i \, : \, i = 1, \dots, k \}$ of $\g{g}_\alpha \ominus \g{m}$ adapted to decomposition~\eqref{equation:decomposition:normal}. Since each $\g{m}_\varphi^\perp$ is known to have constant Kähler angle $\varphi$ for any $\varphi \in \Phi \backslash \{0\}$ by virtue of Proposition~\ref{proposition:kahler:angle}~(\ref{proposition:kahler:angle:2}), we get that $\langle J \xi^\top, J \xi^\top \rangle = \sin^2 \varphi$ for any $\xi \in \g{m}_\varphi^\perp$. This is still true if $\xi \in \g{c}$, as it has constant Kähler angle zero. Using this in Lemma~\ref{lemma:shape:xi}~(\ref{lemma:shape:xi:4}), together with the definition of $\mathcal{D}$ given in~\eqref{equation:derivation:nilsolitons} and Lemma~\ref{lemma:Jacobi:xi}~(\ref{lemma:Jacobi:xi:2}), we deduce~(\ref{lemma:D:2}).
\end{proof}

\begin{proposition}\label{proposition:nilradical}
Let $S$ be a connected Lie subgroup of dimension at least two of the solvable part $AN$ of the Iwasawa decomposition associated with the complex hyperbolic space $\C H^n$. Then, $S$ is a Ricci soliton when considered with the induced metric if and only if one the following conditions holds for the Lie algebra $\g{l}$ of its nilradical $L$:
\begin{enumerate}[{\rm (i)}]	
\item $\g{l} = \g{m}_{\pi/2} \oplus \R (V + t Z)$, where $\g{m}_{\pi/2}$ is a totally real subspace of $\g{g}_\alpha$, $V \in \g{g}_\alpha$ is $\C$-orthogonal to $\g{m}_{\pi/2}$, and $t \in \R$. If $V = 0$ and $t = 0$, then $\dim \g{m}_{\pi/2} \geq 1$. The Lie algebra~$\g{l}$ is abelian and $L$ is isometric to a Euclidean space. \label{proposition:nilradical:1}
\item $\g{l} =  \g{m}_{\varphi} \oplus  \g{m}_{\pi/2} \oplus \g{g}_{2\alpha}$, where $\g{m}_\varphi$ is a non-zero subspace of $\g{g}_\alpha$ of constant Kähler angle $\varphi \in [0, \pi/2)$, and $\g{m}_{\pi/2}$ is a totally real subspace of $\g{g}_\alpha$. The Lie algebra $\g{l}$ is never abelian and thus $L$ is a non-Einstein Ricci soliton. \label{proposition:nilradical:2}
\end{enumerate}
\end{proposition}

\begin{proof}
Since $S$ is a Lie subgroup of dimension at least two of the solvable part $AN$ of the Iwasawa decomposition associated with the complex hyperbolic space $\C H^n$, then the Lie algebra $\g{l}$ of its nilradical is contained in $\g{n}$ by virtue of Lemma~\ref{lemma:nilpotent:nilradical}~(\ref{lemma:nilpotent:nilradical:1}). 

Without loss of generality, we can and will assume that $\g{l} = \g{m} \oplus \R (V + t Z)$, where $\g{m}$ is a real subspace of $\g{g}_\alpha$, $V$ is a vector in $\g{g}_\alpha \ominus \g{m}$, and $t$ a real number. Let $L$ be the connected Lie subgroup of $AN$ with Lie algebra $\g{l}$. Moreover, from Proposition~\ref{proposition:kahler:angle}, we have the orthogonal decomposition $\g{m} = \bigoplus_{\varphi \in \Phi} \g{m}_\varphi$, where $\g{m}_\varphi$ is a subspace of $\g{g}_\alpha$ of constant Kähler angle $\varphi$, for each $\varphi$ in a finite subset $\Phi$ of $[0, \pi/2]$. We will divide this proof depending on whether $\g{l}$ is an abelian Lie algebra or not.

(\ref{proposition:nilradical:1}): \textbf{$\g{l}$ abelian}. Then $\g{m}$ has to be totally real, $\Phi \subset \{\pi/2\}$, by virtue of~\eqref{equation:brackets}. Moreover, $\C V$ must be orthogonal to $\g{m}$. This is trivially true if $V = 0$ or if $\g{m} = 0$. Otherwise, take $Y \in \g{m}$ not orthogonal to $JV$. Then $[V + t Z, Y] = \langle JV, Y\rangle Z \neq 0$ and $\g{l}$ would not be abelian. Recall from Section~\ref{section:preliminaries} that $L$ is diffeomorphic to some Euclidean space. Using this and that $L$ is abelian and thus flat, we deduce that it is isometric to a Euclidean space.

(\ref{proposition:nilradical:2}): \textbf{$\g{l}$ not abelian}. On the one hand, since $[\g{n}, \g{n}] \subset \g{g}_{2\alpha}$ from~\eqref{equation:brackets}, we deduce that $\g{g}_{2\alpha}\subset\g{l}$ and then we can rewrite $\g{l}$ as $\g{l} = \g{m} \oplus \g{g}_{2\alpha}$. On the other hand, $\Phi \neq \{ \pi /2 \}$, since $\g{m}_{\pi/2}$ is totally real and $\g{l}$ is not abelian.

The proof is now based on analyzing~\eqref{definition:derivation} for any $X$, $Y \in \g{l}$. It suffices to take them running in an orthonormal basis of $\g{l}$, as $\mathcal{D}$ is linear. Since $\mathcal{D} \rvert_\g{m}$ and $\mathcal{D} \rvert_{\g{g}_{2\alpha}}$ are (in general different) multiples of the identity by means of Lemma~\ref{lemma:D}, condition~\eqref{definition:derivation} holds trivially for any $X$, $Y \in \g{l}$ such that $[X, Y] = 0$. Thus, we can skip the following cases: $X$, $Y \in \g{m}_{\pi/2}$; $X \in \g{m}_\varphi$ and $Y \in \g{m}_\psi$ for any distinct $\varphi$, $\psi \in \Phi$; $X \in V_\varphi^l$ and $Y \in V_\varphi^m$, for any $\varphi \in \Phi \backslash \{ 0 \}$ and any distinct $l$, $m \in \{ 1, \dots, k_\varphi \}$ (recall $2 k_\varphi = \dim \g{m}_\varphi$); $\C$-orthogonal $X$, $Y \in \g{m}_0$; and $X$ or $Y$ in $\g{g}_{2\alpha}$.

Therefore, we will consider simultaneously the two cases left: $X$, $Y \in V_\varphi^l$, for some $\varphi \in \Phi \cap (0, \pi/2)$ and some $l \in \{ 1, \dots, k_\varphi \}$; and $X$, $Y = JX \in \g{m}_0$. According to Proposition~\ref{proposition:kahler:angle}~(\ref{proposition:kahler:angle:3}), we can and will assume $\langle JX, Y \rangle = \cos \varphi$ for both cases. Using this in~\eqref{equation:brackets}, together with Lemma~\ref{lemma:D}~(\ref{lemma:D:1})-(\ref{lemma:D:2}), we deduce
\begin{equation}\label{equation:nilsoliton}
\begin{aligned}
\mathcal{D}[X, Y] & = \cos \varphi \left( -\frac{1}{4} \dim (\g{g}_\alpha \ominus \g{m}) + \frac{1}{4} \sum_{\psi \in \Phi} \dim \g{m}_\psi \sin^2\psi +c  \right)Z, \\
[\mathcal{D} X, Y] & = [X, \mathcal{D}Y] = \cos \varphi \left(-\frac{1}{4} \dim (\g{g}_\alpha \ominus \g{m}) -\frac{1}{2} \sin^2 \varphi + c\right)Z.
\end{aligned} 
\end{equation}
Thus, combining~\eqref{definition:derivation}, ~\eqref{equation:nilsoliton} and $\varphi \neq \pi/2$, we deduce that $L$ is a nilsoliton if and only if 
\begin{equation}\label{equation:c}
c = \frac{1}{4} \left(\dim (\g{g}_\alpha \ominus \g{m}) +\sum_{\psi \in \Phi} \dim \g{m}_\psi \sin^2\psi\right) + \sin^2  \varphi.
\end{equation}
Since the function $\sin^2$ is injective in $[0, \pi/2]$, it follows from~\eqref{equation:c} that $L$ is a nilsoliton if and only if $\{ \varphi \} \subset \Phi \subset \{ \varphi, \pi/2\}$, for some $\varphi \in [0, \pi/2)$. The last claim in~(\ref{proposition:nilradical:2}) follows from taking into account that $\g{l} =  \g{m}_{\varphi} \oplus  \g{m}_{\pi/2} \oplus \g{g}_{2\alpha}$ is a non-abelian nilpotent Lie algebra, and it does not admit Einstein metrics~\cite[Theorem~2.4]{Mi76}.
\end{proof}

\begin{remark}\label{remark:nilradical:nilsoliton}
Let $L$ be a connected Lie subgroup of the nilpotent part $N$ of the Iwasawa decomposition associated with a complex hyperbolic space. Then, $L$ is its own nilradical. Therefore, Proposition~\ref{proposition:nilradical} can be also regarded as a classification of the connected Lie subgroups of $N$ which are Ricci solitons when they are considered with the induced metric.
\end{remark}

\begin{remark}\label{remark:data:nilsolitons}
Let us assume that we are under the conditions of Proposition~\ref{proposition:nilradical}~(\ref{proposition:nilradical:2}), that is, we consider the Lie subalgebra $\g{l} = \g{m}_\varphi \oplus \g{m}_{\pi/2} \oplus \g{g}_{2\alpha}$ of $\g{n}$, for a certain $\varphi \in [0, \pi/2)$, corresponding to a Lie group $L$. Thus, using~\eqref{equation:derivation:nilsolitons}, Lemma~\ref{lemma:D} with $c$ as in~\eqref{equation:c}, taking into account that $\Phi \subset \{ \varphi, \pi/2 \}$, and recalling that $\dim \g{g}_\alpha = 2n-2$, now \eqref{equation:ricci:l} reads as
\begin{equation}\label{equation:ricci:nonabelian:nilradical}
\Ric^L = -\frac{1}{4} \cos ^2\varphi (\dim \g{m}_\varphi + 4)  \id + (\tr \Ss_B) \ad (B) - \mathcal{D},
\end{equation}
where $(\tr \Ss_B)\ad (B) - \mathcal{D}$ is a derivation of $\g{l}$ not proportional to the identity. Otherwise $L$ would be a non-abelian nilpotent Lie algebra endowed with an left-invariant Einstein metric, contradicting~\cite[Theorem~2.4]{Mi76}. Then, $-(1/4)  \cos ^2\varphi (\dim \g{m}_\varphi + 4)$ is the defining constant of $L$ as an algebraic Ricci soliton. Moreover $\ad(B) \rvert_{\g{m}_\psi} = (1/2) \id$ for $\psi \in \{ \varphi, \pi/2 \}$, and $\ad(B) \rvert_{\g{g}_{2\alpha}} = \id$ by virtue of~\eqref{equation:brackets}, and thus $\ad (B)$ is a derivation of $\g{l}$ which is never proportional to the identity. Finally, from Lemma~\ref{lemma:D} and~\eqref{equation:c} we get 

\begin{align}\label{equation:nilsoliton:after}
\nonumber   \mathcal{D} \rvert_{\g{m}_\varphi} & = \frac{1}{4} \left( \sin^2 \varphi (\dim \g{m}_\varphi +2) + \dim \g{m}_{\pi/2} \right) \id, \\
\mathcal{D} \rvert_{\g{m}_{\pi/2}} & = \frac{1}{4} \left( \sin^2 \varphi (\dim \g{m}_\varphi +4) + \dim \g{m}_{\pi/2}-2 \right)  \id, \\
\nonumber\mathcal{D} \rvert_{\g{g}_{2\alpha}} & = \frac{1}{2} \left( \sin^2 \varphi (\dim \g{m}_\varphi +2) + \dim \g{m}_{\pi/2} \right) \id.
\end{align}
\end{remark}

\section{Extending to solvsolitons and Main Theorem}\label{section:solvsolitons}
This section is devoted to finishing the proof of the Main Theorem. In order to do so, we will basically combine Proposition~\ref{proposition:nilradical} and Lemma~\ref{lemma:nilpotent:nilradical}~(\ref{lemma:nilpotent:nilradical:2}) with a result achieved in~\cite{La11C} that we state below, where Lauret details how to obtain any solvsoliton from its nilradical. 

\begin{proposition}\label{proposition:lauret} \cite[Theorem~4.8]{La11C}
Let $S$ be a solvmanifold with metric Lie algebra $(\g{s}, \langle \cdot, \cdot \rangle)$ and consider the orthogonal decomposition $\g{s} = \g{b} \oplus \g{l}$, where $\g{l}$ is the nilradical of $\g{s}$. Then, $\Ric^S = \bar{c} \id + D$ for some $\bar{c} \in \R$ and some derivation $D$ of $\g{s}$, i.e., $S$ is an algebraic Ricci soliton, if and only if the following conditions hold:
\begin{enumerate}[{\rm (i)}]
\item $\g{l}$ is a nilsoliton with the induced metric, namely $\Ric^L = \bar{c} \id + D_1$, for some derivation $D_1$ of $\g{l}$.\label{proposition:lauret:1}
\item $\g{b}$ is abelian.\label{proposition:lauret:2}
\item  $(\ad X)^t$ is a derivation of $\g{s}$, for any $X \in \g{b}$. \label{proposition:lauret:3}
\item $\langle X, X \rangle = -\frac{1}{4 \bar{c}} \tr ( (\ad X) + (\ad X)^t )^2$, for any $X \in \g{b}$.\label{proposition:lauret:4}
\end{enumerate}
Moreover, $S$ is Einstein if and only if $D_1 = \ad (X) \rvert_{\g{l}}$ for some $X \in \g{b}$.
\end{proposition}
The last claim of the above result, concerning the Einstein condition, follows from~\cite[Proposition~4.3]{La11C}. If $\g{s} = \g{b} \oplus \g{l}$ is the Lie algebra of a solvsoliton constructed from its nilradical $L$ with Lie algebra $\g{l}$ as detailed in Proposition~\ref{proposition:lauret}, we will say that $S$ is an extension of $L$ whose rank is defined as the dimension of~$\g{b}$. Hence, by means of Lemma~\ref{lemma:nilpotent:nilradical}~(\ref{lemma:nilpotent:nilradical:2}), any Lie subgroup of $AN \cong \C H^n$ that is a Ricci soliton with the induced metric will be a rank one extension of its nilradical $L \subset N$. These nilradicals were obtained in Proposition~\ref{proposition:nilradical}.

Now, we state and prove an auxiliary result.

\begin{lemma}\label{lemma:derivation:JU}
Let $\g{l} = \g{m} \oplus \g{g}_{2\alpha}$ and $\g{s} = \R T \oplus \g{l}$  be Lie subalgebras of $\g{a} \oplus \g{n}$, where $\g{m}$ is a non-zero real subspace of $\g{g}_\alpha$, and $T = B + U +x Z$ is orthogonal to $\g{l}$, with $U \in \g{g}_\alpha \ominus \g{m}$ and $x \in \R$. If $\ad (T)^t$ is a derivation of $\g{s}$, then $U$ is $\C$-orthogonal to~$\g{m}$.
\end{lemma}

\begin{proof}
We will assume that $JU$ is not orthogonal to $V$ for some unit $V \in \g{m}$, and we will get a contradiction with the fact that $\ad(T)^t$ is a derivation of $\g{s}$. From~\eqref{equation:brackets} we have $\ad(T) T = 0$, $\ad(T) W = (1/2)W + \langle JU, W\rangle Z$ for any $W \in \g{m}$, and $\ad(T) Z = Z$. Let $\{ V_1, \dots, V_m \}$ be an orthonormal basis of $\g{m}$ with $V_1 = V$. Then $\ad(T)^t T = 0$, $\ad(T)^t W = (1/2)W$ for any $W \in \g{m}$, and $\ad(T)^t Z = Z + \sum_{i=1}^m \langle JU, V_i\rangle V_i$. Thus, condition~\eqref{definition:derivation} does not hold if we take $D = \ad(T)^t$, $X = T$ and $Y = Z$. This is because $\langle \ad(T)^t [T, Z], V \rangle = \langle JU, V \rangle$,  $\langle [\ad(T)^t T, Z], V \rangle = 0$ and $\langle [T, \ad(T)^t Z], V \rangle = (1/2) \langle JU, V\rangle$.
\end{proof}

In the following result, we obtain all the non-nilpotent connected Lie subgroups of $AN$ that are Ricci solitons when endowed with the induced metric. The strategy consists in inspecting, by means of Proposition~\ref{proposition:lauret}, all the possible solvable extensions of the nilsolitons obtained in Proposition~\ref{proposition:nilradical} to solvsolitons realized as Lie subgroups of $AN$. 

\begin{proposition}\label{proposition:solvsolitons}
Let $S$ be a non-nilpotent connected Lie subgroup of the solvable part $AN$ of the Iwasawa decomposition associated with the complex hyperbolic space $\C H^n$. Then, $S$ is a Ricci soliton when considered with the induced metric if and only if one of the following conditions holds for its Lie algebra $\g{s}$:
\begin{enumerate}[{\rm (i)}]
\item $\g{s} = \R (B + U + x Z) \oplus \g{m}_{\pi/2} \oplus \R (V + t Z)$, where $\g{m}_{\pi/2}$ is a totally real subspace of $\g{g}_\alpha$ (non-trivial if both $V$ and $t$ are zero), $U$, $V \in \g{g}_\alpha$ are $\C$-orthogonal to $\g{m}_{\pi/2}$ satisfying $\langle JU, V \rangle = -(t/2)$ when $V \neq 0$, and $t$, $x \in \R$. In this case, $S$ is Einstein if and only if: $V \neq 0$; or $V = 0$, and either $t = 0$ or $\g{m}_{\pi/2}= 0$, but they are not zero~simultaneously.\label{proposition:solvsolitons:1}
		
\item $\g{s} = \R(B + U) \oplus \g{m}_{\varphi} \oplus  \g{m}_{\pi/2} \oplus \g{g}_{2\alpha}$, where $\g{m}_\varphi$ is a non-zero subspace of $\g{g}_\alpha$ of constant Kähler angle $\varphi \in [0, \pi/2)$, $\g{m}_{\pi/2}$ is a totally real subspace of $\g{g}_\alpha$, and $U \in \g{g}_\alpha$ is $\C$-orthogonal to $ \g{m}_{\varphi} \oplus  \g{m}_{\pi/2}$ satisfying~\eqref{equation:U}. In this case, $S$ is Einstein if and only if~$\g{m}_{\pi/2} = 0$. \label{proposition:solvsolitons:2}
\end{enumerate}
\end{proposition}
	
\begin{proof}
Since $S$ is non-nilpotent, then $\dim S \geq 2$ and according to Lemma~\ref{lemma:nilpotent:nilradical}~(\ref{lemma:nilpotent:nilradical:2}) we can write $\g{s} = \R T \oplus \g{l}$, where $\g{l} \subset \g{n}$ is the nilradical of $\g{s}$ (classified in Proposition~\ref{proposition:nilradical}) and $T$ is a vector in $\g{a} \oplus \g{n}$ with non-trivial projection onto $\g{a}$ and orthogonal to $\g{l}$. Now, the proof is based on analyzing in which circumstances~$\g{s}$ satisfies the conditions from~(\ref{proposition:lauret:1}) to~(\ref{proposition:lauret:4}) in Proposition~\ref{proposition:lauret}. Conditions~(\ref{proposition:lauret:1}) and~(\ref{proposition:lauret:2}) are always satisfied, as $\g{l}$ will be taken from Proposition~\ref{proposition:nilradical} and $\g{b} = \R T$ is one-dimensional. At this point, we divide the proof depending on whether the nilradical is abelian or not, according to Proposition~\ref{proposition:nilradical}.

(\ref{proposition:solvsolitons:1}): \textbf{$\g{l}$ abelian.} In this case, any endomorphism of $\g{l}$ is also a derivation and we can write $\Ric^L = 0 = \bar{c} \id + D$ for any $\bar{c} < 0$, by taking $D = -\bar{c} \id$. Thus, condition~(\ref{proposition:lauret:4}) of Proposition~\ref{proposition:lauret} is trivially satisfied, as $\g{b} = \R T$ is one-dimensional. Therefore, we just need to analyze condition~(\ref{proposition:lauret:3}) of Proposition~\ref{proposition:lauret}.

According to Proposition~\ref{proposition:nilradical}~(\ref{proposition:nilradical:1}) and Lemma~\ref{lemma:derivation:JU}, we can write $\g{s} = \R T \oplus \g{m}_{\pi/2} \oplus \R (V + t Z)$, where $T = B + U + xZ$ with $U \in \g{g}_\alpha \ominus \g{m}_{\pi/2}$, $V \in \g{g}_\alpha$ is $\C$-orthogonal to $\g{m}_{\pi/2}$, and $t$, $x \in \R$. Note that $\g{s}$ has to be a Lie subalgebra of $\g{a} \oplus \g{n}$. Since $\g{m}_{\pi/2} \oplus \R (V + t Z)$ is abelian, we just need to check brackets involving the vector $T$. If $ V \neq 0$, from $[T, V + tZ] \in \g{s}$ and~\eqref{equation:brackets} we deduce that $\langle JU, V \rangle = -(t/2)$. If $V \neq 0$, or $V = 0$ and $t=0$, from $[T, W] \in \g{s}$ for any $W \in \g{m}_{\pi/2}$, using~\eqref{equation:brackets}, we deduce that $U$ is $\C$-orthogonal to $\g{m}_{\pi/2}$ (which is still true if $\g{m}_{\pi/2} = 0$). Thus, either $U$ is $\C$-orthogonal to $\g{m}_{\pi/2}$, or $V = 0$ and $t \neq 0$ (so we do not get any information from the Lie algebra condition). However, in this last case, if $S$ is a Ricci soliton, then $\ad(T)^t$ must be a derivation of $\g{s}$ and we use Lemma~\ref{lemma:derivation:JU}. In conclusion: if $S$ is a Ricci soliton, then $U$ is $\C$-orthogonal to $\g{m}_{\pi/2}$.

Now, $\ad(T) T = 0$ and $\ad(T) \rvert_{\g{m}_{\pi/2}} = (1/2) \id$. Moreover, $\ad(T) \rvert_{\R (V + t Z)}$ is $(1/2) \id$ if $V \neq 0$ and $\id$ otherwise. Therefore $\ad(T)^t \rvert_\g{s} = \ad(T) \rvert_\g{s}$, which implies that $\g{s}$ satisfies Proposition~\ref{proposition:lauret}~(\ref{proposition:lauret:3}) and $S$ is a Ricci soliton. This proves the first claim in~(\ref{proposition:solvsolitons:1}).

Finally, from Proposition~\ref{proposition:lauret}, we have that $S$ is Einstein if and only if $\ad(T) \rvert_{\g{m}_{\pi/2} \oplus \R (V + t Z)}$ is proportional to $D = -c \id$. But this happens if and only if: $V \neq 0$; or $V = 0$ and either $t = 0$ or $\g{m}_{\pi/2} = 0$, but they are not simultaneously trivial (as $\dim S \geq 2$). 

(\ref{proposition:solvsolitons:2}): \textbf{$\g{l}$ non-abelian.} According to Proposition~\ref{proposition:nilradical}~(\ref{proposition:nilradical:2}),  Lemma~\ref{lemma:nilpotent:nilradical}~(\ref{lemma:nilpotent:nilradical:2}) and since $\g{g}_{2\alpha}$ is contained in $\g{s}$, we get $\g{s} = \R T \oplus \g{m}_\varphi \oplus \g{m}_{\pi/2} \oplus \g{g}_{2 \alpha}$, where $\g{m}_\varphi$ is a non-zero subspace of $\g{g}_\alpha$ of constant Kähler angle $\varphi \in [0, \pi/2)$, $ \g{m}_{\pi/2}$ is a totally real subspace of $\g{g}_\alpha$, and $T = B+U$, with $U \in \g{g}_\alpha \ominus (\g{m}_\varphi \oplus \g{m}_{\pi/2})$. It is clear that $\g{s}$ is always a Lie subalgebra of $\g{a} \oplus \g{n}$. Put $\g{m} = \g{m}_\varphi \oplus \g{m}_{\pi/2}$. Combining Proposition~\ref{proposition:lauret}~(\ref{proposition:lauret:3}) with Lemma~\ref{lemma:derivation:JU}, we deduce that if $S$ is a Ricci soliton, then $U$ must be $\C$-orthogonal to $\g{m}$. Then, we obtain $\ad(T)T=0$, $\ad(T) \rvert_\g{m}  = (1/2) \id$, and $\ad(T) \rvert_{\g{g}_{2\alpha}} = \id$. Thus $\ad(T)^t \rvert_\g{s} = \ad(T) \rvert_\g{s}$ is a derivation of $\g{s}$. Moreover, we have $\tr \ad^2(T) \rvert_{\g{s}} = (1/4) \dim \g{m} + 1$. Now, using~\eqref{equation:ricci:nonabelian:nilradical} from Remark~\ref{remark:data:nilsolitons}, we obtain that condition~(\ref{proposition:lauret:4}) from Proposition~\ref{proposition:lauret} reads as
\begin{equation}\label{equation:U}
\langle U, U \rangle = \langle T, T \rangle - 1 = \frac{\dim \g{m} +4}{(\dim \g{m}_\varphi +4) \cos ^2(\varphi)}-1.
\end{equation}
Hence, $S$ is a Ricci soliton if and only if $U$ satisfies~\eqref{equation:U}. Let us study in which cases $S$ is Einstein. According to Remark~\ref{remark:data:nilsolitons} and the last claim in Proposition~\ref{proposition:lauret}, $S$ is Einstein if and only $\ad(T) \rvert_{\g{l}}$ is proportional to $((\tr \Ss_B) \ad(B) - \mathcal{D}) \rvert_{\g{l}}$. Recall that $\mathcal{D}$ was explicitly calculated in~\eqref{equation:nilsoliton:after}. Since $\ad(T) \rvert_{\g{l}} = \ad(B) \rvert_{\g{l}}$, then $S$ is Einstein if and only if $\mathcal{D}$ is also proportional to them. Since $\varphi \neq \pi/2$, this happens if and only if $\g{m}_{\pi/2} = 0$, as follows~from~\eqref{equation:nilsoliton:after}.
\end{proof}

We are now in position to prove the main result of this paper.

\begin{proof}[Proof of the Main Theorem]
Let $S$ be a connected Lie subgroup of dimension $\geq 2$ of the solvable part $AN$ of the Iwasawa decomposition associated with a complex hyperbolic space. If $S$ is a Ricci soliton when considered with the induced metric, then it must appear in Proposition~\ref{proposition:nilradical} if it is nilpotent (see Remark~\ref{remark:nilradical:nilsoliton}), or in Proposition~\ref{proposition:solvsolitons} otherwise. In this line, examples in Proposition~\ref{proposition:nilradical}~(\ref{proposition:nilradical:1}) correspond to those in Main Theorem~(\ref{MainTheorem:1}). Examples in Proposition~\ref{proposition:nilradical}~(\ref{proposition:nilradical:2}) give rise to those in Main Theorem~(\ref{MainTheorem:4}). Examples in Proposition~\ref{proposition:solvsolitons}~(\ref{proposition:solvsolitons:1}) correspond to items~(\ref{MainTheorem:2}) or~(\ref{MainTheorem:5}) of the Main Theorem, depending on whether they are Einstein or not, respectively. Finally, examples from Proposition~\ref{proposition:solvsolitons}~(\ref{proposition:solvsolitons:2}) correspond to those in Main Theorem~(\ref{MainTheorem:3}) if they are Einstein ($\g{m}_{\pi/2} = 0$) and to the ones in Main Theorem~(\ref{MainTheorem:6}) otherwise.

Let us justify the second claim of each item of the Main Theorem.

On the one hand, consider the subalgebra $\g{s}$ of item~(\ref{MainTheorem:1}). Since $S$ is abelian and diffeomorphic to a Euclidean space (see Section~\ref{section:preliminaries}), it is isometric to a Euclidean space. This completes the proof of items~(\ref{MainTheorem:1}) and~(\ref{MainTheorem:5}).

On the other hand, consider the Lie algebra $\g{s}$ of item~(\ref{MainTheorem:2}). Assume first $\g{m}_{\pi/2} \neq 0$ or $V \neq 0$. Consider the totally real subspace $\g{r} = \g{m}_{\pi/2} \oplus \R V$ of $\g{g}_\alpha$. Note that the connected Lie subgroup of $AN$ whose Lie algebra is $\R B \oplus \g{r}$ is isometric to a real hyperbolic space and it is a totally geodesic submanifold of $AN$ as well. Now, consider the Lie algebra isomorphism $f \colon \R B \oplus \g{r} \to  \g{s}$ mapping $B$ to $(B+U+xZ)$, $\frac{V}{|V|}$ to $\frac{|B+U+xZ|}{|V+tZ|}(V + tZ)$, and  $f \rvert_{\g{m}_{\pi/2}} =|B+U+xZ| \id$. Note that $f$ is a homothety, as $\langle f(X), f(X) \rangle =|B+U+xZ|^2 \langle X, X \rangle$ for any $X \in \R B \oplus \g{r}$. Since we are working with left-invariant metrics we get that $S$ is isometric to a real hyperbolic space up to scaling. The remaining case corresponds to $\g{s} = \R (B+U) \oplus \R Z$. Similarly, one gets that this Lie algebra is isomorphic and isometric, up to scaling, to $\R B \oplus \R Z$, which corresponds to the Lie algebra of a totally geodesic $\C H^1 \cong \R H^2$ in $AN$. This completes the proof of item~(\ref{MainTheorem:2}). 

Finally, consider the Lie algebra $\g{s}$ given in Proposition~\ref{proposition:solvsolitons}~(\ref{proposition:solvsolitons:2}) and take $\g{m}_{\pi/2}=0$. Let $\g{c}$ be a complex subspace of $\g{g}_\alpha$ of the same dimension as $\g{m}_\varphi$. Note that the connected Lie subgroup of $AN$ with Lie algebra $\g{h} = \g{a} \oplus \g{c} \oplus \g{g}_{2\alpha}$ is isometric to a complex hyperbolic space and it is a totally geodesic submanifold of $AN$. Consider the orthonormal basis of $\g{m}_\varphi$ given by Proposition~\ref{proposition:kahler:angle}~(\ref{proposition:kahler:angle:3}), and an orthonormal basis of $\g{c}$ of the form $\{X_1, J X_1, \dots, X_{k_\varphi}, JX_{k_\varphi}\}$. Then, the unique linear map $f \colon \g{s} \to \g{h}$ satisfying
\begin{equation*}
f(T) = B, \,  f (e_{2l-1}) = \cos \varphi X_l,\,  f (\cos \varphi J e_{2l-1} + \sin \varphi  Je_{2l}) = \cos \varphi J X_l \, \, \,  \text{and } \,  f(Z) = \cos \varphi Z
\end{equation*}
with $T = B+U$ and $l \in \{1, \dots, k_\varphi\}$, is an isomorphism of Lie algebras and an homothety, as $\langle f(X), f(X) \rangle =\cos^2 \varphi \langle X, X \rangle$ for any $X \in \g{s}$,  since $U$ satisfies~\eqref{equation:U}. Thus, we get that $S$ is isometric, up to scaling, to a complex hyperbolic space, which completes the proof of item~(\ref{MainTheorem:3}). Moreover, $f \rvert_{\g{m}_{\varphi} \oplus \g{g}_{2 \alpha}} \colon \g{m}_{\varphi} \oplus \g{g}_{2 \alpha} \to \g{c} \oplus \g{g}_{2 \alpha}  $ is still an isomorphism of Lie algebras and an homothety of constant $\cos^2 \varphi$. Note that the connected Lie subgroup of $AN$ with Lie algebra $\g{c} \oplus \g{g}_{2 \alpha}$ is isomorphic to a Heisenberg group~\cite[Chapter~3]{BTV95}. Since $[\g{m}_{\pi/2}, \g{m}_\varphi] = [\g{m}_{\pi/2}, \g{g}_{2 \alpha}] = 0$, the connected Lie subgroup of $AN$ with a Lie algebra as in item~(\ref{MainTheorem:4}) is isometric, up to scaling, to the product of a Heisenberg group times a Euclidean space (of dimension $\dim \g{m}_{\pi/2}$). This completes the proof of items~(\ref{MainTheorem:4}) and~(\ref{MainTheorem:6}).
\end{proof}

Corollaries~\ref{corollary:real:hyperbolic:space}, ~\ref{corollary:totally:geodesic:symmetric:space} and~\ref{corollary:extension} follow directly from the Main Theorem.

\begin{proof}[Proof of Corollary~\ref{corollary:hopf:isoparametric}]
In order to prove this result, we will see that the only minimal examples of the Main Theorem correspond to one of the following cases: item~(\ref{MainTheorem:2}) with $U = V = 0$, $x = 0$, and either $\g{m}_{\pi/2} = 0$ or $t=0$; item~(\ref{MainTheorem:3}) with $\varphi = 0$ (and thus $U = 0$); and item~(\ref{MainTheorem:5}) with $U = 0$. In the first two cases, $S$ is either a totally geodesic $\R H^k$ or a totally geodesic $\C H^k$ in $\C H^n$, for some $k \leq n$. Therefore, $S$ is the singular orbit of a cohomogeneity one action whose principal orbits are Hopf hypersurfaces on a totally geodesic $\C H^k$ in $\C H^n$~\cite{DDS:adv} (where $k$ can be taken as $n$ if $S \cong \C H^k$). In the last case, Main Theorem~(\ref{MainTheorem:5}) with $U = 0$, $S$ is a minimal (not totally geodesic) submanifold. Indeed: if $n =2$, then $S$ is the so-called Lohnherr hypersurface; if $n >2$, then $S$ is the focal set of an isoparametric family of hypersurfaces that are homogeneous if and only if $\g{g}_{\alpha} \ominus \g{m}_{\pi/2}$ is totally real~\cite{DDS:adv}.

In order to address this proof, we will need to calculate several shape operators, so we will make extensive use of Lemma~\ref{lemma:nablas} without mentioning it explicitly.

Assume first that $S$ is under the conditions of items~(\ref{MainTheorem:1}) or~(\ref{MainTheorem:4}) of the Main Theorem. Then $B$ is a normal vector to $S$ in $AN$. Now, we get $\Ss_B X = (1/2) X$ for any $X \in \g{g}_\alpha \cap \g{s}$. Moreover, $\Ss_B Z = Z$ if $\g{s}$ corresponds to item~(\ref{MainTheorem:4}) and 
\[
\langle \Ss_B (V+tZ),  (V + t Z) \rangle= (1/2)\langle V, V \rangle +t^2 
\]
if $\g{s}$ corresponds to item~(\ref{MainTheorem:1}). Therefore, $\tr \Ss_B > 0 $ and $S$ is not minimal under the conditions of items~(\ref{MainTheorem:1}) or~(\ref{MainTheorem:4}) of the Main Theorem.
 
Let us focus on the rest of the cases assuming first $U \neq 0$. Then, $\xi = |U|^2 B - U$ is a normal vector to $S$ in $AN$. If $X$ is a unit vector in $\g{s} \cap \g{g}_\alpha$, then $2 \langle \Ss_\xi X, X \rangle = |U|^2$. Moreover, if $Z \in \g{s}$, we get $\langle \Ss_\xi Z, Z \rangle = |U|^2$, and if $B+U+xZ \in \g{s}$, we get
\begin{equation}\label{equation:shape:minimal}
\langle \Ss_\xi (B+U+xZ), (B+U+xZ) \rangle  = (1/2) |U|^2 ( |U|^2+1+2x^2). 
\end{equation}
From these considerations and taking $x = 0$ in~\eqref{equation:shape:minimal}, we conclude that $S$ is not minimal under the conditions of items~(\ref{MainTheorem:3}), (\ref{MainTheorem:5}) or~(\ref{MainTheorem:6}) of the Main Theorem if $U \neq 0$. Note that in item~(\ref{MainTheorem:6}) the vector $U$ is always different from zero, as follows from~\eqref{equation:U} and the fact that $\g{m}_\varphi$, $\g{m}_{\pi/2}$ are non-trivial. As explained above, if $U = 0$ in items~(\ref{MainTheorem:3}) or~(\ref{MainTheorem:5}) of the Main Theorem, then $S$ is a totally geodesic Lie subgroup of $AN$.

Thus, we focus on item~(\ref{MainTheorem:2}) of the Main Theorem, which is the remaining case to analyze. Using~\eqref{equation:levi-civita}, we get  
\begin{equation}\label{equation:shape:minimal:case2}
\langle \Ss_{\xi} (V + t Z), (V + t Z) \rangle = (1/2) (|U|^2 |V|^2 +t \langle JV, U \rangle +2 t^2 |U|^2).
\end{equation}
If $U \neq 0$, from equation~\eqref{equation:shape:minimal} and its previous two lines, equation~\eqref{equation:shape:minimal:case2} and $\langle JV, U \rangle = -\langle V, JU \rangle = (t/2)$ if $ V \neq 0$, we get that $S$ is not minimal. Thus, put $U = 0$. Hence, from the subalgebra condition one gets $V = 0$ or $t = 0$. Recall that $V = 0$ implies that either $t = 0$ or $\g{m}_{\pi/2} = 0$ (but not simultaneously). Thus, we have to analyze these two remaining cases: $\g{s} = \R (B+xZ) \oplus \g{w}_{\pi/2}$, where $\g{w}_{\pi/2}$ is a totally real subspace of $\g{g}_\alpha$; and $\g{s} = \R (B+xZ) \oplus \g{g}_{2\alpha}$. Since decompositions in the Main Theorem are orthogonal, the second case is $\g{s} = \R B \oplus \g{g}_{2\alpha}$, and then $S$ is isometric to a totally geodesic $\C H^1$ in $\C H^n$. Thus, $S$ is the singular orbit of cohomogeneity one action whose principal curvatures are Hopf hypersurfaces. For the remaining case, take the normal vector $\eta = x B - Z$. Then $\langle \Ss_\eta (B+ xZ), (B+ xZ)\rangle = x (x^2+1)$ and $\langle \Ss_\eta X, X \rangle = (1/2)x$ for any $X \in \g{w}_{\pi/2}$. Thus, minimality implies $x = 0$ and $S$ is a totally geodesic real hyperbolic space.

Finally, note that for any real subspace $\g{m}$ of $\g{g}_\alpha$, the connected Lie subgroup of $AN$ with Lie algebra $\g{s} = \R B \oplus \g{m} \oplus \g{g}_{2\alpha}$ leads to the (minimal) focal set of an isoparametric family of hypersurfaces~\cite{DDS:adv}. However, we have obtained that these focal sets are Ricci solitons with the induced metric only when $\g{m}$ is complex or totally real.
\end{proof}

Corollary~\ref{corollary:einstein:minimal:totallygeodesic} follows directly from Corollary~\ref{corollary:hopf:isoparametric} and the Main Theorem.

\begin{proof}[Proof of Corollary~\ref{corollary:minimal:n}]
We just need to see that the examples in Main Theorem~(\ref{MainTheorem:4}) give rise to minimal submanifolds of $N$. Note that the Levi-Civita connection of $N$ is achieved by projecting orthogonally the Levi-Civita connection of $AN$ onto $\g{n}$. Hence, the minimality follows from the last claim of Lemma~\ref{lemma:shape:xi}.	
\end{proof}

\begin{remark}
Corollary~\ref{corollary:minimal:n} is not true if we remove the non-flatness condition. Indeed, let $L$ be the connected Lie subgroup of $N$ with Lie algebra as in Main Theorem~(\ref{MainTheorem:1}), with $V$ of unit length and $t \neq 0$. Then, one can see that the trace of the shape operator of $L$ as a submanifold of $N$ with respect to the unit normal vector $JV$ is different from zero.
\end{remark}

We finish the paper with a remark concerning the congruence classes of the examples of the Main Theorem.

\begin{remark}\label{remark:congruence}
Let $\g{s}$ be a Lie subalgebra of $\g{a} \oplus \g{n}$. Since $\g{a} \oplus \g{n}$ is a complex vector space, then $\g{s}$ can be decomposed in a unique way as a $\C$-orthogonal sum of subspaces of $\g{a}\oplus \g{n}$ of constant Kähler angle, as follows from Proposition~\ref{proposition:kahler:angle}. In Table~\ref{table:a}, we detail the (different) Kähler angles of the subspaces stemming from such decomposition and their corresponding dimensions, for the Lie subalgebra $\g{s}$ of $\g{a} \oplus \g{n}$, depending on the item of the Main Theorem. For items~(\ref{MainTheorem:3}) and~(\ref{MainTheorem:6}) we have used the conditions on $U$ obtained in the Main Theorem.
\begin{table}[h]
\begin{center}	
\renewcommand*{\arraystretch}{1.2}	\scriptsize
\begin{tabular}{|c|c|c|c|c|c|}
\hline
Case  & (\ref{MainTheorem:1}) & (\ref{MainTheorem:3}) & (\ref{MainTheorem:4})  & (\ref{MainTheorem:5}) & (\ref{MainTheorem:6})  \\ [1ex]
\hline
Kähler angles & $\pi/2$ & $\varphi$ & $\varphi$, $\frac{\pi}{2}$ &  $\frac{\pi}{2}$, $\arccos(1+|U|^2)^{-1}$ & $\varphi$, $\frac{\pi}{2}$, $\arccos\bigl(\frac{(\dim \g{m}_\varphi +4) \cos ^2(\varphi)}{\dim \g{m}_\varphi +\dim \g{m}_{\pi/2} + 4}\bigr)$   \\[2ex]		
\hline
Dimensions & $\dim \g{s}$ & $\dim \g{s}$ & $\dim \g{m}_\varphi$, $\dim \g{m}_{\pi/2}+1$ &  $\dim \g{m}_{\pi/2}$, $2$ &  $\dim \g{m}_\varphi$, $\dim \g{m}_{\pi/2}$, $2$   \\[2ex]		
\hline
\end{tabular}
\bigskip
\caption{Decomposition of $\g{s}$, from Main Theorem, into subspaces of $\g{a} \oplus \g{n}$ of constant Kähler angle.}
\label{table:a}
\end{center}
\end{table}	
Now, let $\g{s}_1 = \R (B+U_1) \oplus \g{m}_1 \oplus \g{g}_{2\alpha}$ and $\g{s}_2 = \R (B+U_2) \oplus \g{m}_2 \oplus \g{g}_{2\alpha}$ be two subalgebras of $\g{a} \oplus \g{n}$, with $\g{m}_1$, $\g{m}_2 \subset \g{g}_\alpha$, both under the conditions of one (but the same) of the items from~(\ref{MainTheorem:3}) to~(\ref{MainTheorem:6}) of the Main Theorem. In item~(\ref{MainTheorem:4}) we just assume that $\R (B+U_1)$ and $\R (B+U_2)$ do not appear in the definition. Let $S_1$ and $S_2$ be the connected Lie subgroups of $AN$ whose Lie algebras are $\g{s}_1$ and $\g{s}_2$, respectively. 

Let $\phi$ be an isometry of $AN$ such that $\phi (S_1) = S_2$. Without loss of generality, we assume that $\phi$ preserves the neutral element of $AN$. Thus, it follows that $\phi$ is a unitary or anti-unitary transformation of $\g{a} \oplus \g{n}$ and maps subspaces of constant Kähler angle to subspaces of the same constant Kähler angle. Therefore, according to Table~\ref{table:a}, $\g{m}_1$ and $\g{m}_2$ must have the same Kähler angles with the same multiplicities. Note that this implies $|U_1| = |U_2|$ by means of the Main Theorem.

Conversely, assume that $\g{m}_1$ and $\g{m}_2$ have the same Kähler angles with the same multiplicities. Note that $K_0$, the connected Lie subgroup of $G = SU(1,n)$ with Lie algebra $\g{k}_0$, is isomorphic to $U(n-1)$ and acts on $\g{g}_\alpha$ (identified with $\C^{n-1}$) in the standard way. Since $U_1$ and $U_2$ are $\C$-orthogonal to $\g{m}_1$ and $\g{m}_2$, respectively, and of the same length, then there exists $k \in K_0$ such that $\Ad (k) (\g{m}_1 \oplus \R U_1) = \g{m}_2 \oplus \R U_2$, as follows from~\cite[Remark 2.10]{DDK17}. Since $\Ad (k) B = B$ and $\Ad (k) Z = Z$, then $\Ad (k) \g{s}_1 = \g{s}_2$ and $k(S_1) = S_2$. This gives the congruence classes from items~(\ref{MainTheorem:3}) to~(\ref{MainTheorem:6}), both included, in the Main Theorem.

Assume that $\g{s}$ is under the conditions of item~(\ref{MainTheorem:1}), or of item~(\ref{MainTheorem:2}) with $t = 0$, of the Main Theorem. Let $S$ be the connected Lie subgroup of $AN$ with Lie algebra $\g{s}$. Then $S$ is a homogeneous CR-submanifold (complex-real). These submanifolds were classified in~\cite[Theorem A, Theorem B]{DDP23} up to congruence, so one can consult the congruence classes there. The congruence classes can be completely determined if one analyzes item~(\ref{MainTheorem:2}) from the Main Theorem with $t \neq 0$. However, calculations are long and involved and we content ourselves with stating that they provide infinitely non-congruent examples, as follows from the proof of the Main Theorem.
\end{remark}

\enlargethispage{2\baselineskip}


\begin{thebibliography}{99}
	
\bibitem{AlKi} D.~V.~Alekseevskii, B.~N.~ Kimel'fel'd: Structure of homogeneous Riemannian spaces with zero Ricci curvature. \emph{Functional Anal. Appl.} \textbf{9} (1975), no. 2, 97--102.

\bibitem{Be89} J.~Berndt: Real hypersurfaces with constant principal curvatures in complex hyperbolic space. \emph{J. Reine Angew. Math.} \textbf{395} (1989), 132--141.

\bibitem{BB01} J.~Berndt, M.~Br\"uck: Cohomogeneity one actions on hyperbolic spaces. \emph{J.\ Reine Angew.\ Math.} \textbf{541} (2001), 209--235.

\bibitem{Jurgen} J.~ Berndt, S.~Console, C.~Olmos: \emph{Submanifolds and holonomy. Second edition.} Monographs and Research Notes in Mathematics. CRC Press, Boca Raton, FL, 2016.

\bibitem{BeTa07}  J.~Berndt, H.~Tamaru: Cohomogeneity one actions on noncompact symmetric spaces of rank one. \emph{Trans. Amer. Math. Soc.} \textbf{359} (2007), no. 7, 3425--3438.

\bibitem{BTV95} J.~Berndt, F. Tricerri, L. Vanhecke: \emph{Generalized Heisenberg groups and Damek-Ricci harmonic spaces.}
Lecture Notes in Mathematics \textbf{1598}, Springer-Verlag, Berlin, 1995.

\bibitem{BoLa21} C.~B\"ohm, R.~Lafuente: Non-compact Einstein manifolds with symmetry. \emph{J.\ Amer.\ Math.\ Soc.} \textbf{36} (2023), 591--651.

\bibitem{CeRy} T.~E.~Cecil, P.~J.~Ryan: \emph{Geometry of hypersurfaces.} Springer Monographs in Mathematics. Springer, New York, 2015.

\bibitem{Fi38} A.~Fialkow: Hypersurfaces of a space of constant curvature. \emph{Ann. of Math.} \textbf{39} (1938), 762--785.

\bibitem{DaTo} M.~Dajczer, R.~Tojeiro: \emph{Submanifold Theory. Beyond an Introduction.} Universitext, Springer, New York, 2019.

\bibitem{DD11} J.~C.~Díaz-Ramos, M.~Domínguez-Vázquez: Non-Hopf real hypersurfaces with constant principal curvatures in complex space forms. \emph{Indiana Univ. Math. J.} \textbf{60} (2011), 859--882.

\bibitem{DDH}J.~C.~Díaz-Ramos, M.~Domínguez-Vázquez, T.~Hashinaga: Homogeneous Lagrangian foliations on complex space forms. \emph{Proc. Amer. Math. Soc.} \textbf{151} (2023), 823--833.

\bibitem{DDK17} J.~C.~Díaz-Ramos, M.~Domínguez-Vázquez,  A.~Kollross: Polar actions on complex hyperbolic spaces. \emph{Math.\ Z.}\ \textbf{287} (2017), no.~3--4, 1183--1213.

\bibitem{DDP23} J.~C.~Díaz-Ramos, M.~Domínguez-Vázquez, O.~Pérez-Barral: Homogeneous CR submanifolds of complex hyperbolic spaces. \emph{Publ. Mat.} \textbf{67} (2023), 891--912.

\bibitem{DDS:adv} J.~C.~D\'iaz-Ramos, M.~Dom\'inguez-V\'azquez, V.~Sanmart\'in-L\'opez:
Isoparametric hypersurfaces in complex hyperbolic spaces. \emph{Adv.\ Math.}\ \textbf{314} (2017), 756--805. 

\bibitem{DST20} M.~Dom\'inguez-V\'azquez, V.~Sanmart\'in-L\'opez, H.~Tamaru: Codimension one Ricci soliton subgroups of solvable Iwasawa groups. \emph{J. Math. Pures Appl.} \textbf{152} (2021), 69--93.

\bibitem{DiOl} A.~Di Scala, C.~Olmos: The geometry of homogeneous submanifolds of hyperbolic
space. \emph{Math. Z.} \textbf{237} (2001), no. 1, 199--209.

\bibitem{Gold} W.~Goldman: \emph{Complex hyperbolic geometry}. Oxford Mathematical Monographs, Clarendon Press, Oxford, 1999.

\bibitem{Ha82} R.~Hamilton: Three-manifolds with positive Ricci curvature. \emph{J. Differential Geom.} \textbf{17} (1982), no. 2, 255--306.

\bibitem{He98} J.~Heber: Noncompact homogeneous Einstein spaces. \emph{Invent. math.} \textbf{133} (1998), 279--352.

\bibitem{Iv93} T.~Ivey: Ricci solitons on compact three-manifolds. \emph{Differential Geom. Appl.} \textbf{3} (1993), no. 4, 301--307.

\bibitem{Jab:arxiv} M.~Jablonski: Einstein solvmanifolds as submanifolds of symmetric spaces. arXiv:1810.11077.

\bibitem{K} A.~W.~Knapp: \emph{Lie groups beyond an introduction. Second edition}. Progress in Mathematics, 140, Birkh\"auser Boston, Inc., Boston, MA, 2002.

\bibitem{La01} J.~Lauret: Ricci soliton homogeneous nilmanifolds. \emph{Math. Ann.} \textbf{319} (2001), 715--733.

\bibitem{La11C} J.~Lauret: Ricci soliton solvmanifolds. \emph{J. Reine Angew. Math.} \textbf{650} (2011), 1--21.

\bibitem{La09} J.~Lauret: Einstein solvmanifolds and nilsolitons. New developments in Lie theory and
geometry.\emph{ Contemp. Math.}, vol. \textbf{491}, Amer. Math. Soc, Providence, RI, 2009, pp. 1--35.

\bibitem{Mi76} J.~Milnor: Curvature of left-invariant metrics on Lie groups. \emph{Adv. Math.} \textbf{21} (1976), 293--329.

\bibitem{Na10} A.~Naber: Noncompact shrinking four solitons with nonnegative curvature. \emph{J. Reine Angew. Math.} \textbf{645} (2010), 125--153.

\bibitem{NiPa21} Y.~Nikolayevsky, J.~Park: Einstein hypersurfaces in irreducible symmetric spaces.\emph{ Ann. Mat. Pura Appl. (4)} \textbf{202} (2023), no. 4, 1719--1751.

\bibitem{PeWy09} P.~Petersen, W.~Wylie: On gradient Ricci solitons with symmetry. \emph{Proc. Amer. Math. Soc.} \textbf{137} (2009), 2085--2092.

\bibitem{Ta11} H.~Tamaru: Parabolic subgroups of semisimple Lie groups and Einstein solvmanifolds. \emph{Math. Ann.} \textbf{351} (2011), no.~1, 51--66.

\bibitem{S22} V.~Sanmartín-López: Codimension one Ricci soliton subgroups of nilpotent Iwasawa groups. arXiv:2212.04743.

\bibitem{Wolter} T. H.~Wolter: Einstein Metrics on solvable groups.\emph{ Math. Z.} \textbf{206} (1991), no. 3, 457--471.
\end{thebibliography}
\end{document}